\documentclass[a4paper,11pt,reqno]{amsart}

\usepackage{amsfonts}
\usepackage{amssymb}
\usepackage{amscd}
\usepackage{amsthm}
\usepackage{a4wide}
\usepackage{mathrsfs}
\usepackage[applemac]{inputenc}
\usepackage{amsmath}
\usepackage{amssymb}
\usepackage{amsthm}
\usepackage{textcomp}
\usepackage{graphicx}
\usepackage{enumerate}
\usepackage{mathrsfs}
\usepackage{frcursive}
\usepackage{tikz}
\usepackage[cyr]{aeguill}
\usepackage{xspace}
\usepackage{hyperref}
\usepackage{appendix}

\newtheorem{defn}{Definition}[section]
\newtheorem{lemma}[defn]{Lemma}
\newtheorem{prop}[defn]{Proposition}
\newtheorem{theo}[defn]{Theorem}
\newtheorem{coro}[defn]{Corollary}

\newtheorem{rk}[defn]{Remark}

\def\Ric{\mathop{\rm Ric}\nolimits}
\def\R{\mathop{\rm R}\nolimits}

\def\Rm{\mathop{\rm Rm}\nolimits}
\def\tr{\mathop{\rm tr}\nolimits}

\def\diam{\mathop{\rm diam}\nolimits}
\def\vol{\mathop{\rm vol}\nolimits}
\def\eucl{\mathop{\rm eucl}\nolimits}
\def\dim{\mathop{\rm dim}\nolimits}
\def\vol{\mathop{\rm Vol}\nolimits}

\def\inj{\mathop{\rm inj}\nolimits}

\def\div{\mathop{\rm div}\nolimits}

\def\A{\mathop{\rm A}\nolimits}

\def\AVR{\mathop{\rm AVR}\nolimits}
\def\Id{\mathop{\rm Id}\nolimits}

\def\cat#1{{\mathfrak{#1}}}

\def\Ric{\mathop{\rm Ric}\nolimits}

\def\Rm{\mathop{\rm Rm}\nolimits}
\def\tr{\mathop{\rm tr}\nolimits}

\def\diam{\mathop{\rm diam}\nolimits}
\def\vol{\mathop{\rm vol}\nolimits}
\def\eucl{\mathop{\rm eucl}\nolimits}
\def\dim{\mathop{\rm dim}\nolimits}
\def\vol{\mathop{\rm Vol}\nolimits}

\def\inj{\mathop{\rm inj}\nolimits}

\def\div{\mathop{\rm div}\nolimits}

\def\A{\mathop{\rm A}\nolimits}

\def\AVR{\mathop{\rm AVR}\nolimits}
\def\Id{\mathop{\rm Id}\nolimits}

\def\cat#1{{\mathfrak{#1}}}

\def\Sym{\mathop{\rm Sym}\nolimits}

\def\crit{\mathop{\rm Crit}\nolimits}

\newcommand{\nfa}{\arrowvert\nabla f\arrowvert}

\title[Asymptotic estimates and compactness of expanding gradient Ricci solitons]{Asymptotic estimates and compactness of expanding gradient Ricci solitons}
\author[Alix Deruelle]{Alix Deruelle}
\address[Alix Deruelle]{Mathematics Institute, University of Warwick, Gibbet Hill Rd, Coventry, West Midlands CV4 7AL}
\email{A.Deruelle@warwick.ac.uk}

\begin{document}
\begin{abstract}
We first investigate the asymptotics of conical expanding gradient Ricci solitons by proving sharp decay rates to the asymptotic cone both in the generic and the asymptotically Ricci flat case. We then establish a compactness theorem concerning nonnegatively curved expanding gradient Ricci solitons.
  \end{abstract}
\maketitle

\section{Introduction}
The Ricci flow, introduced by Hamilton in the early eighties, is a non linear heat equation on the space of metrics of a given manifold, modulo the action of diffeomorphisms and homotheties. Therefore, one expects the Ricci flow to smooth out singular geometric structures : we consider here the smoothing of metric cones over smooth compact manifolds by expanding self-similarities of the Ricci flow. More precisely, we recall that an expanding gradient Ricci soliton is a triplet $(M,g,\nabla f)$ where $f:M\rightarrow \mathbb{R}$ is a smooth function such that $$\nabla^2f=\Ric(g)+\frac{g}{2}.$$
At least formally, one can associate a Ricci flow solution by defining $g(\tau):=(1+\tau)\phi_{\tau}^*g$ where $(\phi_{\tau})_{\tau>-1}$ is the one parameter family of diffeomorphisms generated by the vector field $-\nabla^gf/(1+\tau)$.

Let's describe a first class of non trivial examples, i.e. non Einstein, of Ricci expanders discovered by Bryant [Chap.1,\cite{Cho-Lu-Ni-I}] : it is a one parameter family of rotationally symmetric metrics $(g_c)_{c>0}$ on $\mathbb{R}^n$ whose asymptotic cones (in the Gromov sense) are the metric cones $(C(\mathbb{S}^{n-1}),dr^2+(cr)^2g_{\mathbb{S}^{n-1}})$, with $c\in\mathbb{R}_+^*$, where $(\mathbb{S}^{n-1},g_{\mathbb{S}^{n-1}})$ is the Euclidean sphere of curvature $1$. The metrics $g_c$ have positive curvature operator if $c<1$ and are negatively curved if $c>1$. The metric cone $(C(\mathbb{S}^{n-1}),dr^2+(cr)^2g_{\mathbb{S}^{n-1}})$ is both the asymptotic cone of $(M,g_c)$ and the singular initial condition of the associated Ricci flow at time $\tau=-1$. Such a phenomenon has been generalized at least in two directions. In the positively curved case, Schulze and Simon \cite{Sch-Sim} have shown that most of the asymptotic cones of non collapsed Riemannian manifolds are smoothed out by Ricci gradient expanders. On the other hand, Chen and the author \cite{Che-Der} proved that any Ricci gradient expander with quadratic curvature decay has a unique asymptotic cone $(C(X),dr^2+r^2g_X)$, where $X$ is a smooth compact manifold and $g_X$ is a $C^{1,\alpha}$ metric for any $\alpha\in(0,1)$.

The main purpose of this article consists in understanding how the geometry of the link $(X,g_X)$ is reflected globally or asymptotically on the Ricci soliton. We describe first the asymptotics of Ricci expanders. We recall the definition of a gradient Ricci expander being asymptotically conical.
\begin{defn}
An expanding gradient Ricci soliton $(M,g,\nabla f)$ is asymptotically conical to $(C(X),g_{C(X)}:=dr^2+r^2g_X,r\partial r/2)$ if there exists a compact $K\subset M$, a positive radius $R$ and a diffeomorphism $\phi:M\setminus K\rightarrow C(X)\setminus B(o,R)$ such that
\begin{eqnarray}
&&\sup_{\partial B(o,r)}\arrowvert\nabla^k(\phi_*g-g_{C(X)})\arrowvert_{g_{C(X)}}=\textit{O}(f_{k}(r)),\quad \forall k\in \mathbb{N},\label{class-def-asy-con}\\
&&f(\phi^{-1}(r,x))=\frac{r^2}{4},\quad\forall (r,x)\in C(X)\setminus B(o,R), \label{cond-exp-1}
\end{eqnarray}
where $f_k(r)=\textit{o}(1)$ as $r\rightarrow+\infty$.
\end{defn}
Condition (\ref{class-def-asy-con}) is one of the classical definitions of asymptotically conical metrics (e.g. in \cite{Sie-PHD} with Schauder norms instead) and condition (\ref{cond-exp-1}) taken from \cite{Cho-Egs} reflects the compatibility with the expanding structure.

We are mainly interested in the following situations : either the convergence is polynomial at rate $\tau$, i.e. $f_{k}(r)=r^{-\tau-k}$ for some positive $\tau$ and any nonnegative integer $k$, either the convergence is exponential at rate $\tau$, i.e. $f_k(r)=r^{-\tau+k}e^{-r^2/4}$ for any nonnegative integer $k$.

The generic case, i.e. without any other additional assumptions on the asymptotic cone, was implicitly considered in \cite{Che-Der} : we proved that the convergence was polynomial at rate $\tau=2$. This rate is sharp regarding the existing examples : the Bryant examples or the Ricci expanders built by Feldman-Ilmanen-Knopf \cite{Fel-Ilm-Kno} on the negative line bundle $(L^{-k},g_{k,p})_{k,p}$ with $k,p\in\mathbb{N}^*$ coming out of the cone $(C(\mathbb{S}^{2n-1}/\mathbb{Z}_k),i\partial\bar{\partial}(\arrowvert z\arrowvert^{2p}/p))$, where $\mathbb{Z}_k$ acts on $\mathbb{C}^n$ by rotations. 

If one imposes natural bounds on covariant derivatives of the curvature operator in a rescaled way, i.e.
$$\limsup_{x\rightarrow+\infty}r_p(x)^{2+i}\arrowvert\nabla^i\Rm(g)\arrowvert<+\infty,$$
for any nonnegative integer $i$, one gets $C^{\infty}$ convergence to the asymptotic cone : see \cite{Sie-PHD}. Actually, for both the convenience of the reader and the sake of coherence of the paper, we give a different proof under weaker assumptions : theorem \ref{asy-con-gen}. Now, if the curvature tensor of the link $X$ satisfies some elliptic constraints, we expect to get faster convergence to the asymptotic cone or at least a faster convergence of some tensor made out of the curvature tensor of the Ricci expander. We mainly consider two cases :\\

\begin{itemize}
\item If $(C(X^{n-1}),g_{C(X^{n-1})})$ is Ricci flat, we get sharp estimates of the Ricci curvature decay : $$\Ric(g)=\textit{O}(r^{2-n}e^{-r^2/4}),$$ which improves previous results by Siepmann \cite{Sie-PHD} : see theorem \ref{dec-Ricci-flat-cone}.\\
\item If $(X,g_X)$ is Einstein with constant scalar curvature different from $(n-1)(n-2)$, then $$T:=\Ric(g)-\frac{\R_g}{n-1}\left(g-\frac{\nabla f}{\arrowvert\nabla f\arrowvert}\otimes\frac{\nabla f}{\arrowvert\nabla f\arrowvert}\right)=\textit{O}(r^{-4}).$$
This tensor $T$ reflects the conical geometry at infinity and has already been proved useful in other contexts : \cite{Bre-Rot-3d}. This is proved in theorems \ref{dec-ric-c^3} and \ref{dec-ric-c^4}. Such a result holds in dimension $3$ without any assumptions on the link : see proposition \ref{est-sge-3d}.
\end{itemize}
 Finally, we also focus on how bounds on asymptotic covariant derivatives of the Ricci (resp. scalar) curvature imply corresponding bounds on asymptotic covariant derivatives of the full (resp. Ricci) curvature (resp. in dimension $3$) : see the proof of theorem \ref{asy-con-gen} and proposition \ref{est-sge-3d}.  
 
 We emphasize the fact that all these estimates are quantitative and are not so useful regarding compactness questions which is the second aspect of this article. Indeed, all the previous results do not control all the topology of the Ricci expander. The first result about (pre)compactness of Riemannian metrics goes back to Cheeger and Gromov. As we are mainly interested in the $C^{\infty}$ topology, we state Hamilton's compactness theorem with bounds on covariant derivatives of the curvature tensor and bounded diameter \cite{Ham-Com}  :
 \begin{theo}[Hamilton]\label{comp-ham-static}
 \begin{eqnarray}
 \cat{M}(n,D,v,(\Lambda_k)_{ k\geq0}):=&&\{(N^n,h)\quad \mbox{closed}\quad | \diam(h)\leq D\quad;\quad\vol(N,h)\geq v;\\
 &&\arrowvert\nabla^k\Rm(h)\arrowvert\leq\Lambda_k,\forall k\in\mathbb{N}\},
 \end{eqnarray}
 where $D$, $v$ and $(\Lambda_k)_{k\in\mathbb{N}}$ are positive real numbers, is compact in the $C^{\infty}$ topology.
 \end{theo}
If one drops the diameter bound, a similar statement holds in the $C^{\infty}$ \textit{pointed} convergence for complete Riemannian manifolds : then, the lower volume bound concerns the volume of geodesic balls of a fixed radius, say $1$.
  The conditions on the covariant derivatives of the curvature tensor can be actually replaced by bounds on the Ricci curvature and are given for free when an elliptic or parabolic equation is assumed to hold : \cite{Ham-Com} in the the Ricci flow case and \cite{Has-Mul-Com} for Ricci shrinkers and the references therein. In our setting, as $(C(X),g_{C(X)})$ is the initial condition of a non linear heat equation, we expect heuristically at least that if $(X,g_X)$ belongs to a compact set of metrics $\cat{M}$ then so will be the set of expanders with asymptotic cone  $(C(X),g_{C(X)})_{(X,g_X)\in\cat{M}}$.
 
 We are able to prove compactness of Ricci expanders with non negative Ricci curvature. In this case, the potential function is a strictly convex function with quadratic growth, in particular, the topology is that of $\mathbb{R}^n$ and the level sets of the potential function are diffeomorphic to $\mathbb{S}^{n-1}$. More precisely, if $\crit(f)$ denotes the set of critical points of a function $f$,
 
\begin{theo}\label{Compactness-I}
The class
\begin{eqnarray*}
\cat{M}^0_{Exp}(n,\lambda_0,(\Lambda_k)_{k\geq 0})&:=&\{\mbox{$(M^n,g,\nabla f,p)$ normalized expanding gradient Ricci soliton s.t.} 
\\&&\Ric(g)\geq 0\quad;\quad\crit(f)=\{p\}\quad;\quad\sup_{M^n}\arrowvert \Rm(g)\arrowvert\leq\lambda_0\quad ;\\
&&\limsup_{x\rightarrow+\infty}r_p(x)^{2+k}\arrowvert\nabla^k\Rm(g)\arrowvert(x)\leq\Lambda_k,\quad \forall k\geq0 \}
\end{eqnarray*}
is compact in the pointed $C^{\infty}$ topology. \\

Moreover, let a sequence $(M_i,g_i,\nabla f_i,p_i)_i$ be in $\cat{M}^0_{Exp}(n,\lambda_0,(\Lambda_k)_{k\geq0})$. Then there exists a subsequence converging in the pointed $C^{\infty}$ topology to an expanding gradient Ricci soliton $(M_{\infty},g_{\infty},\nabla f_{\infty},p_{\infty})$ in $\cat{M}^0_{Exp}(n,\lambda_0,(\Lambda_k)_{k\geq 0})$ whose asymptotic cone $(C(X_{\infty}),g_{C(X_{\infty})},o_{\infty})$ is the limit in the Gromov-Hausdorff topology of the sequence of the asymptotic cones \\$(C(X_i),g_{C(X_i)},o_i)_i$ with $C^{\infty}$ convergence outside the apex. \\
\end{theo}

\begin{rk}
\begin{itemize}
\item Theorem \ref{Compactness-I} does not assume a uniform lower bound for the rescaled volume : this is actually a consequence of the proof of theorem \ref{Compactness-I}.
\item It would be interesting to allow different underlying topologies in theorem \ref{Compactness-I} as in the Feldman-Ilmanen-Knopf examples \cite{Fel-Ilm-Kno}.
\item We could also have assumed a bound on the Ricci curvature of the form $\Ric(g)\geq-\delta g$ with $\delta\in[0,1/2)$ giving again the same properties for the potential function, the statement of theorem \ref{Compactness-I} should be modified slightly then and would have been less concise. 
\end{itemize}
\end{rk}

We state an equivalent version of theorem \ref{Compactness-I} in terms of initial conditions, i.e. in terms of asymptotic cones lying in a compact set that is more in the spirit of compactness results in the setting of solutions to (non linear) heat equations :
\begin{theo}\label{Compactness-I-bis}
The class
\begin{eqnarray*}
&&\cat{M}^1_{Exp}(n,\lambda_0,(\Lambda_k)_{k\geq 0}):=\{\mbox{$(M^n,g,\nabla f,p)$ normalized expanding gradient Ricci soliton} \\&&\mbox{smoothly converging to its asymptotic cone $(C(X),dr^2+r^2g_X,o)$ s.t.} \quad
\Ric(g)\geq 0\quad;\\&&\crit(f)=\{p\}\quad;\quad\sup_{M^n}\arrowvert \Rm(g)\arrowvert\leq\lambda_0\quad ;\quad(X,g_X)\in\cat{M}(n-1,\pi,c(n,\Lambda_0),(\Lambda_k)_{k\geq 0}) \}
\end{eqnarray*}
satisfies the same conclusion of theorem \ref{Compactness-I}. 
\end{theo}

\begin{rk}
\begin{itemize}
\item The proof of theorem \ref{Compactness-I-bis} is exactly the same as the one of theorem \ref{Compactness-I}. Indeed, since the convergence to the cone is smooth, the invariants $$\limsup_{x\rightarrow+\infty}r_p(x)^{2+k}\arrowvert\nabla^k\Rm(g)\arrowvert$$ can be bounded from above by the supremum of the covariant derivatives of the curvature of the metric of the section of the asymptotic cone up to order $k$ by  Gauss formula.
\item The reason why the lower bound on the volume in the statement of theorem \ref{Compactness-I-bis} depends on $n$ and $\Lambda_0$ is given by the proof of theorem \ref{Compactness-I}. Actually, the nonnegativity of the Ricci curvature implies that $X$ is diffeomorphic to a codimension one sphere endowed with a metric $g_X$ with a positive lower bound on the Ricci curvature depending only on the dimension and a curvature bound $\Lambda_0$ : one can then invoke a non-collapsing theorem due to Petrunin and Tuschmann \cite{Pet-Tus-Pi_2} to bound from below the volume of $g_X$.
\end{itemize}
\end{rk}

An additional difficulty in proving theorems \ref{Compactness-I} and \ref{Compactness-I-bis} consists in inverting limits, i.e. formally speaking,
\begin{eqnarray}
\lim_{t\rightarrow+\infty}\lim_{i\rightarrow+\infty}(M_i,t^{-2}g_i,p_i)=\lim_{i\rightarrow+\infty}\lim_{t\rightarrow+\infty}(M_i,t^{-2}g_i,p_i),
\end{eqnarray}
where $(M_i,g_i,p_i)_i$ is a sequence of expanding gradient Ricci solitons as in theorem \ref{Compactness-I}. The main step for proving this inversion is to derive a priori estimates of differential inequalities with quadratic non linearities. Indeed, the norm of the curvature tensor $\arrowvert\Rm(g)\arrowvert=:u$ satisfies $$\Delta_fu\geq-u-c(n)u^2,$$
as soon as $\arrowvert\Rm(g)\arrowvert$ does not vanish : see corollary \ref{a-prio-bd-curv-tensor}. 

The only tool used massively all along the proofs of asymptotic and a priori global estimates is the maximum principle for functions, this contrasts with the shrinking case where integral estimates appear to be more useful : \cite{Has-Mul-Com}. Finally, on the one hand, Theorem \ref{Compactness-I} can be interpreted as the extension of Hamilton's  compactness theorem for non negatively curved expanding gradient Ricci solitons with conical initial condition, on the other hand, the motivation of theorem \ref{Compactness-I} comes equally from an analogous statement (but which is much harder to prove) in the setting of conformally compact Einstein manifolds where the notion of asymptotic cone $(C(X),g_{C(X)})$ is replaced by the one of conformal infinity $(X,[g_X])$ : see \cite{And-Con-Inf} for details.\\

In the course of the proof of the a priori estimates, we are able to partially answer questions $(9)$ and $(11)$ asked in [Section $7.2$, Chap. $9$,\cite{Ben}].\\

Concerning a lower bound on the rescaled scalar curvature, we get :
\begin{prop}\label{inf-bound-scal}
Let $(M,g,\nabla f)$ be a normalized expanding gradient Ricci soliton with non negative Ricci curvature. Then,
$$\left(\sqrt{\mu(g)}+\frac{r_p(x)}{2}\right)^2\R_g(x)\geq \liminf_{y\rightarrow+\infty}\frac{r_p(y)^2}{4}\R_g(y),\quad \forall x\in M,$$
where $p$ is the unique critical point of the potential function $f$ and $\mu(g)$ is the entropy.
\end{prop}
The definition of the entropy of a Ricci expander can be found in lemma \ref{id-EGS}.

Note that if the convergence to the asymptotic cone $(C(X),g_{C(X)})$ of such a Ricci expander is at least $C^2$ and if the scalar curvature $\R_{g_X}$ is positive then according to proposition \ref{inf-bound-scal}, the rescaled scalar curvature has a definite positive lower bound.\\

Finally, recall that a Riemannian manifold $(M^n,g)$ is \textit{Ricci pinched} if there exists some $\epsilon\in(0,1]$ such that $$\Ric(g)\geq \frac{\epsilon}{n}\R_gg.$$
We are able to prove the following :
\begin{prop}\label{ric-pin}
Let $(M^n,g,\nabla f)$ be an expanding gradient Ricci soliton with non negative scalar curvature which is Ricci pinched. Then,
\begin{enumerate}
\item $(M^n,g,\nabla f)$ is asymptotically conical to a Ricci flat metric cone $\left(C(\mathbb{S}^{n-1}),g_{C(\mathbb{S}^{n-1})},r\partial_r/2\right)$ at exponential rate.\\ 
\item If $\dim M=4$, $(M^4,g,\nabla f)$ is isometric to the Gaussian soliton $(\mathbb{R}^4,\eucl,\arrowvert\cdot\arrowvert^2/4)$.\\
\item If $g$ is sufficiently pinched, i.e. if $\epsilon\in[\epsilon(n),1]$ where $\lim_{n\rightarrow+\infty}\epsilon(n)=1-2^{-1/2}$, then $(M^n,g,\nabla f)$ is isometric to the Gaussian soliton.
\end{enumerate}
\end{prop}

Note that we don't assume a bound on the full curvature tensor : we derive such bounds by the sole assumption of being Ricci pinched.\\

Previous works on this question were done in the following cases :\\
\begin{itemize}
\item If $(M,g)$ has a sufficiently pinched curvature tensor (implying the nonnegativity of the sectional curvature) : \cite{Che-Zhu-Pin}
\item If $(M,g)\times \mathbb{R}^2$ has pinched isotropic curvature (implying the nonnegativity of the sectional curvature) : \cite{Bre-Sch-Cur-Sph} extending \cite{Che-Zhu-Pin}, proving the non existence of non flat non compact  such pinched metrics without any expanding structure. 
\item If $\dim M=3$ : \cite{Ma-Ric-Pin-3d} by essentially using the same argument of the proof of proposition \ref{ric-pin} in dimension $4$.
\end{itemize}

We give a brief outline of the organization of the paper. \\

In section \ref{sol-equ-sec}, we first recall and (re)-prove curvature identities on expanding gradient Ricci solitons, then we focus on precise estimates of the potential function (proposition \ref{pot-fct-est}). Secondly, we (re)-prove local estimates in the spirit of Shi both for the curvature tensor (lemma \ref{elliptic}) and general solutions of weighted elliptic equations (lemma \ref{loc-est-tensor-egs}). Finally, we prove a priori estimates for sub solutions of weighted elliptic equations : lemma \ref{lemm-cruc}, though abstract, is at the core of most results of this paper. \\

Section \ref{asy-est-sec} is devoted to the proof of asymptotic estimates on conical expanding gradient Ricci solitons : we treat separately the generic (theorem \ref{asy-con-gen}) and the asymptotically Ricci flat case (theorem \ref{dec-Ricci-flat-cone}). Then we focus on Einstein constraints at infinity : proposition \ref{est-sge-3d} in dimension $3$ and theorems \ref{dec-ric-c^3} and \ref{dec-ric-c^4} for higher dimensions, then we end this section by proving proposition \ref{ric-pin}. \\

Finally, in section \ref{comp-theo-sec}, we first prove general global a priori estimates : propositions \ref{a-priori-reaction} and \ref{curv-op-max-ppe-infty}, corollaries \ref{a-prio-bd-curv-tensor} and \ref{scal-max-ppe-infty}. This leads us directly to the proof of two compactness theorems : theorems \ref{Compactness-I} and \ref{Compactness-II}.\\

\textbf{Acknowledgements.}
 I would like to thank Panagiotis  Gianniotis, Felix Schulze and Peter Topping for all these fruitful conversations we had on Ricci expanders.\\

The author is supported by the EPSRC on a Programme Grant entitled ‘Singularities of Geometric Partial Differential Equations’ (reference number EP/K00865X/1).

\section{Soliton equations and rough estimates}\label{sol-equ-sec}
\subsection{Algebraic identities on Ricci expanders}
We start with a couple of definitions.\\

\begin{itemize}
\item Let $(M,g)$ be a Riemannian manifold and let $w:M\rightarrow\mathbb{R}$ be a smooth function on $M$. The \textit{weighted laplacian} (with respect to $w$) is 
$\Delta_{w}T:=\Delta T+\nabla_{\nabla w}T$,
where $T$ is a tensor on $M$.

\item An expanding gradient Ricci soliton is said \textit{normalized} if $\int_Me^{-f}d\mu_g=(4\pi)^{n/2}$ (whenever it makes sense).
\end{itemize}
The next lemma gathers well-known Ricci soliton identities together with the (static) evolution equations satisfied by the curvature tensor.
\begin{lemma}\label{id-EGS}
Let $(M^n,g,\nabla f)$ be a normalized expanding gradient Ricci soliton. Then the trace and first order soliton identities are :
\begin{eqnarray}
&&\Delta f = \R_g+\frac{n}{2}, \label{equ:1} \\
&&\nabla \R_g+ 2\Ric(g)(\nabla f)=0, \label{equ:2} \\
&&\arrowvert \nabla f \arrowvert^2+\R_g=f+\mu(g), \label{equ:3}\\
&&\div\Rm(g)(Y,Z,T)=\Rm(g)(Y,Z, \nabla f,T),\label{equ:4}
\end{eqnarray}
for any vector fields $Y$, $Z$, $T$ and where $\mu(g)$ is a constant called the entropy.\\

The evolution equations for the curvature operator, the Ricci tensor and the scalar curvature are :
\begin{eqnarray}
&& \Delta_f \Rm(g)+\Rm(g)+\Rm(g)\ast\Rm(g)=0,\label{equ:5}\\
&&\Delta_f\Ric(g)+\Ric(g)+2\Rm(g)\ast\Ric(g)=0,\label{equ:6}\\
&&\Delta_f\R_g+\R_g+2\arrowvert\Ric(g)\arrowvert^2=0,\label{equ:7}
\end{eqnarray}
where, if $A$ and $B$ are two tensors, $A\ast B$ denotes any linear combination of contractions of the tensorial product of $A$ and $B$.
\end{lemma}

\begin{proof}
See [Chap.$1$,\cite{Cho-Lu-Ni-I}] for instance.

\end{proof}

The next lemma is valid for any gradient Ricci soliton and only uses the second Bianchi identity :
 
\begin{lemma}\label{id-op-prel}
Let $(M^n,g,\nabla f)$ be a gradient Ricci soliton. Then
\begin{eqnarray*}
\nabla_{\nabla f}\Rm(g)(W,X,Y,Z)&=&\nabla_W\div\Rm(g)(Y,Z,X)-\nabla_{X}\div\Rm(g)(Y,Z,W)\\
&&+\Rm(g)(X,\nabla^2 f(W),Y,Z)-\Rm(g)(W,\nabla^2 f(X),Y,Z).
\end{eqnarray*}
\end{lemma}

\begin{proof}
By the second Bianchi identity,
\begin{eqnarray*}
\nabla_{\nabla f}\Rm(g)(W,X,Y,Z)=-\nabla_W\Rm(g)(X,\nabla f,Y,Z)-\nabla_{X}\Rm(g)(\nabla f,W,Y,Z).
\end{eqnarray*}
Now,
\begin{eqnarray*}
\nabla_W\Rm(g)(X,\nabla f,Y,Z)&=&W\cdot\Rm(g)(X,\nabla f,Y,Z)-\Rm(g)(\nabla_WX,\nabla f,Y,Z)\\
&&-\Rm(g)(X,\nabla^2 f(W),Y,Z)-\Rm(g)(X,\nabla f,\nabla_WY,Z)\\
&&-\Rm(g)(X,\nabla f,Y,\nabla_WZ)\\
&=&-\nabla_W\div\Rm(g)(Y,Z,X)-\Rm(g)(X,\nabla^2 f(W),Y,Z).
\end{eqnarray*}
\end{proof}

We pursue this section by computing the evolution of the norm of the covariant derivatives of the curvature along the Morse flow generated by the potential function in terms of the covariant derivatives of the Ricci curvature.

\begin{prop}\label{id-op}
Let $(M^n,g,\nabla f)$ be an expanding gradient Ricci soliton. 
\begin{enumerate}
\item Then,\begin{eqnarray*}
\nabla_{\nabla f}\arrowvert\Rm(g)\arrowvert^2+4\nabla^2 f(\Rm(g),\Rm(g))&=&-4\tr(\star\rightarrow\langle\nabla_{\star}\div\Rm(g)(\cdot,\cdot,\cdot),\Rm(g)(\cdot,\cdot,\cdot,\star)\rangle)\\
&=&-4\nabla_i\div\Rm(g)_{abc}\Rm(g)_{abci},
\end{eqnarray*}
where $\nabla^2 f(\Rm(g),\Rm(g)):=\langle \Rm(g)(\cdot,\nabla^2f(\cdot),\cdot,\cdot),\Rm(g)(\cdot,\cdot,\cdot,\cdot)\rangle.$\\
\item (Commutation identities) For any tensor $T$, and any positive integer $k$,
\begin{eqnarray}
\begin{bmatrix}\nabla_{\nabla f},\nabla^k\end{bmatrix} T&=&-\frac{k}{2}\nabla^kT+\sum_{i=0}^k\nabla^{k-i}\Ric(g)\ast\nabla^iT,\label{comm-1}\\
\begin{bmatrix}\Delta,\nabla^k\end{bmatrix}T&=&\nabla^kT\ast\Ric(g)+\sum_{i=0}^{k-1}\nabla^{k-i}\Rm(g)\ast\nabla^iT.\label{comm-2}
\end{eqnarray}
\item For any integer $k\geq 1$,\begin{eqnarray*}
\nabla_{\nabla f}\arrowvert\nabla^k\Rm(g)\arrowvert^2+(2+k)\arrowvert\nabla^k\Rm(g)\arrowvert^2&=&\nabla^{k+2}\Ric(g)\ast\nabla^k\Rm(g)\\&&+\Ric(g)\ast\nabla^k\Rm(g)^*2\\
&&+\sum_{i=1}^{k}\nabla^i\Ric(g)\ast\nabla^{k-i}\Rm(g)\ast\nabla^k\Rm(g).
\end{eqnarray*}
\end{enumerate}
\end{prop}

\begin{proof}
By lemma \ref{id-op-prel},
\begin{eqnarray*}
\nabla_{\nabla f}\arrowvert\Rm(g)\arrowvert^2&=&2\nabla_{\nabla f}\Rm(g)_{ijkl}\Rm(g)_{ijkl}\\
&=& 2(\nabla_i\div\Rm(g)_{klj}-\nabla_j\div\Rm(g)_{kli}\\
&&+\Rm(g)_{j\nabla^2f(i)kl}-\Rm(g)_{i\nabla^2f(j)kl})\Rm(g)_{ijkl}\\
&=&-4\nabla_i\div\Rm(g)_{jkl}\Rm(g)_{jkli}-4\Rm(g)_{i\nabla^2f(j)kl}\Rm(g)_{ijkl}.\\
\end{eqnarray*}
The commutation identities can be proved by induction on $k$.\\

Concerning the third identity, according to (\ref{comm-1}) and proposition \ref{id-op},
\begin{eqnarray*}
&&\nabla_{\nabla f}\arrowvert\nabla^k\Rm(g)\arrowvert^2=2<\nabla^k\nabla_{\nabla f}\Rm(g),\nabla^k\Rm(g)>+2<[\nabla_{\nabla f},\nabla^k]\Rm(g),\nabla^k\Rm(g)>\\
&&=-(2+k)\arrowvert\nabla^k\Rm(g)\arrowvert^2+\nabla^{k+2}\Ric(g)\ast\nabla^k\Rm(g)+\Ric(g)\ast\nabla^k\Rm(g)^*2\\
&&+\sum_{i=1}^k\nabla^{k-i}\Ric(g)\ast\nabla^i\Rm(g)\ast\nabla^k\Rm(g).
\end{eqnarray*}

\end{proof}

For an expanding gradient Ricci soliton $(M^n,g,\nabla f)$, we define $v:M^n\rightarrow\mathbb{R}$ by $$v(x):=f(x)+\mu(g)+n/2,$$ for $x\in M^n$. Then, $v$ enjoys the following properties.

\begin{prop}\label{pot-fct-est}
Let $(M^n,g,\nabla f)$ be a non Einstein expanding gradient Ricci soliton.
\begin{eqnarray}
&&\Delta_fv=v,\label{equ:8} \\
&v&>\arrowvert\nabla v\arrowvert^2.\label{inequ:1}
\end{eqnarray}
Assume $\Ric(g)\geq 0$ and assume $(M^n,g,\nabla f)$ is normalized. Then $M^n$ is diffeomorphic to $\mathbb{R}^n$ and
\begin{eqnarray}
&&v\geq \frac{n}{2}>0.\label{inequ:2}\\
&&\frac{1}{4}r_p(x)^2+\min_{M^n}v\leq v(x)\leq\left(\frac{1}{2}r_p(x)+\sqrt{\min_{M^n}v}\right)^2,\quad \forall x\in M^n,\label{inequ:3}\\
&&\AVR(g):=\lim_{r\rightarrow+\infty}\frac{\vol B(q,r)}{r^n}>0,\quad\forall q\in M^n,\label{inequ:avr}\\
&&-C(n,V_0,R_0)\leq\min_{M^n}f\leq 0\quad;\quad\mu(g)\geq\max_{M^n}\R_g\geq 0,\label{inequ:ent}
\end{eqnarray}
where $V_0$ is a positive number such that $\AVR(g)\geq V_0$, $R_0$ is such that $\sup_{M^n}\R_g\leq R_0$ and $p\in M^n$ is the unique critical point of $v$.
\end{prop}

\begin{rk}
\begin{itemize}
\item (\ref{inequ:avr}) is due to Hamilton : [Chap. $9$ ; \cite{Ben}].
\item In \cite{Car-Ni} , it is proved that $\mu(g)\geq 0$ with equality if and only if it is isometric to the Euclidean space by studying the linear entropy along the heat kernel. Their argument works for Riemannian manifolds with nonnegative Ricci curvature. We give here a simpler but maybe less enlightening proof in the setting of gradient Ricci expanders.
\end{itemize}
\end{rk}

\begin{proof}
(\ref{equ:8}) comes from adding equations (\ref{equ:1}) and (\ref{equ:3}). (\ref{inequ:1}) comes from the maximum principle at infinity established in \cite{Pig-Rim-Set} ensuring that for a non trivial expanding gradient Ricci soliton, one has $\Delta v=\Delta f>0$.\\

Now, if $\Ric(g)\geq0$, the soliton equation ensures that the potential function $f$ (or $v$) satisfies $\nabla^2f\geq(g/2)$ which implies that $v$ is a proper strictly convex function. By Morse-type arguments, one can show that $M$ is diffeomorphic to $\mathbb{R}^n$ and the lower bound of (\ref{inequ:3}) is achieved by integrating the previous differential inequality involving the Hessian of $v$. The upper bound of (\ref{inequ:3}) holds more generally since it comes from (\ref{inequ:1}). \\

Now, for the convenience of the reader, we reprove briefly (\ref{inequ:avr}) with the help of the sublevel set of the potential function. Indeed, with the help of the coarea formula, the rescaled volume of the level sets $t\rightarrow \vol\{f=t\}t^{-(n-1)/2}$ is non decreasing for $t$ larger than $\min_{M^n}f$. By the coarea formula again, it implies directly that the rescaled volume of the sublevel sets $t\rightarrow \vol\{f\leq t\}t^{-n/2}$ is bounded from below by a positive constant. As the potential function and the distance function to the square are comparable by (\ref{inequ:3}), the volume ratio $r^{-n}\vol B(p,r)$ is bounded from below by a positive constant for large radii $r$ hence for any radius by the Bishop-Gromov theorem.\\

Finally, we show first that the scalar curvature is bounded. Indeed, if $(\phi_t)_t$ is the Morse flow associated to $f$, then
\begin{eqnarray*}
\partial_t\R_g=-2\Ric(g)\left(\frac{\nabla f}{\arrowvert\nabla f\arrowvert},\frac{\nabla f}{\arrowvert\nabla f\arrowvert}\right)\leq 0.
\end{eqnarray*}
In particular, it means that $t\rightarrow\max_{f=t}\R_g$ is non increasing and therefore, by continuity of the scalar curvature, we get
\begin{eqnarray*}
\sup_{M^n}\R_g\leq\max_{f=\min_{M^n}f}\R_g=\R_g(p),
\end{eqnarray*}
where $\crit(f)=\{p\}$.

Now, as $\nabla^2f\geq g/2$, 
 \begin{eqnarray*}
f(x)&\leq& \frac{r_p(x)^2}{4}+\min_{M^n}f+\sqrt{\min_{M^n}f+\mu(g)}\frac{r_p(x)}{2}\\
&\leq& \frac{r_p(x)^2}{4}+\min_{M^n}f+\sqrt{\max_{M^n}\R_g}\frac{r_p(x)}{2},
\end{eqnarray*}
for any $x\in M^n$ since $\min_{M^n}f+\mu(g)=\max_{M^n}\R_g$. As $f$ is normalized, we have
\begin{eqnarray*}
(4\pi)^{n/2}&=&\int_{M^n}e^{-f}d\mu(g)\geq\left(nV_0e^{-\min_{M^n}f}\right)\int_0^{+\infty}e^{-r^2/4-\sqrt{\max_{M^n}\R_g}r}r^{n-1}dr\\
&\geq&e^{-\min_{M^n}f}C(n,V_0,R_0),
\end{eqnarray*}
for some positive constant $C(n,V_0,R_0)$. Hence a lower bound on $\min_{M^n} f$. By using a lower bound on $f$ of the form
$$\min_{M^n}f+\frac{r_p(x)^2}{4}\leq f(x),$$
for any $x\in M^n$, one gets by a similar argument using the Bishop-Gromov theorem that $\min_{M^n}f\leq 0$ since $$\int_{\mathbb{R}^n}e^{-\|x\|^2/4}d\mu(\eucl)(x)=(4\pi)^{n/2}.$$
As a consequence of the previous estimates, we get for free that $\mu(g)\geq\max_{M^n}\R_g\geq 0$ without using \cite{Car-Ni} and also an upper bound for $\mu(g)$ as expected : $\mu(g)\leq C(n,V_0,R_0)$. Therefore, if $\mu(g)=0$, one has $\R_g=0$ which implies that $\Ric(g)=0$ by equation (\ref{equ:7}). \cite{Pet-Wyl-Rig} shows then that $(M^n,g,\nabla f)$ is isometric to the Gaussian soliton $(\mathbb{R}^n,\eucl,r\partial_r/2)$. Indeed, the potential function satisfies $\nabla^2f=g/2$ which means that $f(x)=\min_{M^n}f+r_p^2(x)/4$ for any $x\in M^n$, where $p$ is the only critical point of $f$. In particular, it means that $r_p^2$ is smooth on $M^n$ and that $\Delta r_p^2=2n$ on a Riemannian manifold with nonnegative Ricci curvature : the (infinitesimal) Bishop-Gromov theorem then implies that $(M^n,g)$ is Euclidean.

\end{proof}

\subsection{Local Shi's estimates}

We derive now some baby Shi's estimates for expanding gradient Ricci solitons for large radii. The main thing here is to estimate the influence of the supremum of the curvature operator. 

\begin{lemma}\label{elliptic}
If $(M^n,g,\nabla f)$ is an expanding gradient Ricci soliton such that $\Ric(g)\geq -(n-1)K$ on $B(p,r)$ with $K\geq 0$, then, for any $k\geq 0$ and $r\geq 1$,
\begin{eqnarray*}
\sup_{B(p,r/2)}\arrowvert\nabla^k \Rm(g)\arrowvert\leq C\left(n,\sup_{A(p,r/2,r)}\frac{\arrowvert \nabla f\arrowvert}{r},k,Kr^2\right)\sup_{B(p,r)}\arrowvert\Rm(g)\arrowvert(1+\sup_{B(p,r)}\arrowvert\Rm(g)\arrowvert^{k/2}).
\end{eqnarray*}
\end{lemma}

\begin{proof}

Along the proof, $c$ is a constant which only depends on $n$ but can vary from a line to another. We only give the proof for $k=1$.\\

Using the commutation identities of proposition \ref{id-op}, one can prove that
\begin{eqnarray*}
\Delta_f\arrowvert\nabla\Rm(g)\arrowvert^2\geq2\arrowvert\nabla^2\Rm(g)\arrowvert^2-\left(\frac{3}{2}+c(n)\arrowvert\Rm(g)\arrowvert\right)\arrowvert\nabla\Rm(g)\arrowvert^2.
\end{eqnarray*}
The coefficient $3/2$ has no importance in this proof.

Following \cite{Shi-Def} and \cite{Car-Rig-Old}, define $U:=\arrowvert\Rm(g)\arrowvert^2$, $V:=\arrowvert\nabla\Rm(g)\arrowvert^2$ and $W:=\arrowvert\nabla^2\Rm(g)\arrowvert^2$. Consider the function $F:=(U+a)V$ where $a$ is a positive constant to be defined later. $F$ satisfies,
\begin{eqnarray*}
\Delta_f F&=&V(\Delta_f U)+(U+a)(\Delta_f V)+2\langle\nabla U,\nabla V\rangle\\
&\geq&V(2V-(c\arrowvert\Rm(g)\arrowvert+2) U)+(U+a)(2W-(c\arrowvert\Rm(g)\arrowvert+2) W)\\&-&8U^{1/2}VW^{1/2}\\
&\geq&2V^2+2(U+a)W-16 UW-V^2-(c\arrowvert\Rm(g)\arrowvert+4) F\\
&=&V^2+2W(a-7U)-(c\arrowvert\Rm(g)\arrowvert +4)F,
\end{eqnarray*}
where, in the third inequality, we use $2ab\leq a^2/\alpha+\alpha b^2$, for positive numbers $a$ and $b$ with $\alpha=4$. If $a:= 7\sup_{B(p,r)}U$, one gets on $B(p,r)$,
\begin{eqnarray}\label{ineg-max}
\Delta_f F\geq V^2-(c\arrowvert\Rm(g)\arrowvert +4) F.
\end{eqnarray}

Let $\phi:M^n\rightarrow[0,1]$ be a smooth positive function with compact support defined by $\phi(x):=\psi(r_p(x)/r)$ where $r>0$ and $\psi:[0,+\infty[\rightarrow[0,1]$ is a smooth positive function satisfying
\begin{eqnarray*}
\psi\arrowvert_{[0,1/2]}\equiv 1,\quad \psi\arrowvert_{[1,+\infty[}\equiv 0,\quad \psi'\leq 0,\quad \frac{\psi'^2}{\psi}\leq c,\quad \psi''\geq -c.
\end{eqnarray*}

 Let $G:=\phi F$. Then $G$ satisfies,
 \begin{eqnarray*}
\Delta_f G&=&\phi(\Delta_f F)+2\langle\nabla \phi,\nabla F\rangle+F(\Delta_f \phi).
\end{eqnarray*}

At a point where $G$ attains its maximum,
\begin{eqnarray*}
\nabla G=F\nabla \phi+\phi\nabla F=0,\quad \Delta G\leq 0.
\end{eqnarray*}

Multiplying (\ref{ineg-max}) by $\phi$, we get at  a point where $G$ attains its maximum,
\begin{eqnarray*}
0\geq(\phi V)^2-\phi G(c\arrowvert\Rm(g)\arrowvert +4)-2G\frac{\arrowvert\nabla\phi\arrowvert^2}{\phi}+G(\Delta \phi+\langle \nabla f,\nabla \phi\rangle).
\end{eqnarray*}
Now, in any case, $a\phi V\leq G\leq2a\phi V$. Therefore, if $a>0$,
\begin{eqnarray*}
0\geq\left(\frac{G}{8a}\right)^2+G\left[-\phi(c\arrowvert\Rm(g)\arrowvert +4)-2\frac{\arrowvert\nabla\phi\arrowvert^2}{\phi}+\Delta_f \phi\right].
\end{eqnarray*}
So that if $G$ does not vanish identically on $B(p,r)$,
\begin{eqnarray*}
\phi V\leq ca\left[\phi(c\arrowvert\Rm(g)\arrowvert +4)+2\frac{\arrowvert\nabla\phi\arrowvert^2}{\phi}-\Delta_f \phi\right].
\end{eqnarray*}
Now,
\begin{eqnarray*}
\nabla \phi=\frac {\psi'}{r}\nabla r_p,\quad\Delta \phi=\frac{\psi''}{r^2}+\frac{\psi'}{r}\Delta r_p.
\end{eqnarray*}
Hence,
\begin{eqnarray*}
2\frac{\arrowvert\nabla\phi\arrowvert^2}{\phi}-\Delta_f \phi=\frac{1}{r^2}\left[\frac{2\psi'^2}{\psi}-\psi''\right]-\frac{\psi'}{r}(\langle \nabla f,\nabla r_p\rangle+\Delta r_p).
\end{eqnarray*}
On the other hand, $\Ric(g)\geq-(n-1)K$ on $B(p,r)$, with $K\geq 0$, by comparison theorem,
\begin{eqnarray*}
\Delta r_p&\leq& (n-1)K^{1/2}\coth(K^{1/2}r_p)\\
&\leq&\frac{n-1}{r_p}(1+K^{1/2}r_p)\leq\frac{2(n-1)}{r}(1+K^{1/2}r) \quad\mbox{on $B(p,r)\backslash B(p,r/2)$.}
\end{eqnarray*}
 Therefore, by the very definition of $\psi$,
 \begin{eqnarray*}
&&\sup_{B(p,r/2)}\arrowvert\nabla\Rm(g)\arrowvert\leq\\
&& c\sup_{B(p,r)}\arrowvert\Rm(g)\arrowvert\left[\sup_{B(p,r)}\arrowvert\Rm(g)\arrowvert+1+\frac{\sup_{A(p,r/2,r)}\arrowvert \nabla f\arrowvert}{r}+\frac{1}{r^2}(1+K^{1/2}r)\right]^{1/2}.
\end{eqnarray*}

\end{proof}

Finally, we consider local Shi's estimates for solutions of the weighted laplacian with a potential depending on the curvature.

\begin{lemma}[Local covariant derivatives estimates]\label{loc-est-tensor-egs}
Let $(M^n,g,\nabla f)$ be an expanding gradient Ricci soliton such that $\Ric(g)\geq -(n-1)K$ on $B(p,r)$ with $K\geq 0$. Let $T$ be a $C^2$ tensor satisfying
\begin{eqnarray}
\Delta_fT=-\lambda T+\Rm(g)\ast T,\label{evol-equ-gen-tens}
\end{eqnarray}
for some $\lambda\in\mathbb{R}$.

Then, for any $k\geq 0$ and $r\geq 1$,
\begin{eqnarray*}
\sup_{B(p,r/2)}\arrowvert\nabla^kT\arrowvert&\leq&C\left(n,\lambda,\sup_{A(p,r/2,r)}\frac{\arrowvert \nabla f\arrowvert}{r},k,Kr^2\right)\sup_{B(p,r)}\arrowvert T\arrowvert(1+\sup_{B(p,r)}\arrowvert\Rm(g)\arrowvert^{k/2})
\end{eqnarray*}
\end{lemma}

\begin{proof}
The proof is essentially the same as for lemma \ref{elliptic}. We only give the major steps for $k=1$ and $k=2$.

By the evolution equation for $T$ given by (\ref{evol-equ-gen-tens}), we estimate,
\begin{eqnarray*}
\Delta_f\arrowvert T\arrowvert^2&\geq&2\arrowvert\nabla T\arrowvert^2-(c(n)\arrowvert\Rm(g)\arrowvert+2\lambda)\arrowvert T\arrowvert^2.\\
\Delta_f\arrowvert\nabla T\arrowvert^2&\geq&2\arrowvert\nabla^2 T\arrowvert^2-(c(n)(\arrowvert\Rm(g)\arrowvert+\arrowvert\nabla\Rm(g)\arrowvert^{2/3})+2\lambda+1)\arrowvert\nabla T\arrowvert^2\\
&&-c\arrowvert\nabla\Rm(g)\arrowvert^{4/3}\arrowvert T\arrowvert^2.
\end{eqnarray*}
Adapting the proof of lemma \ref{elliptic}, with the same notations, one gets
\begin{eqnarray*}
\Delta_f F\geq V^2-c(\lambda,n)(1+(\arrowvert\Rm(g)\arrowvert+\arrowvert\nabla\Rm(g)\arrowvert^{2/3}))F-c(n)\arrowvert\nabla\Rm(g)\arrowvert^{4/3} a^2.
\end{eqnarray*}
We conclude by following the same procedure and by invoking the results of lemma \ref{elliptic}.

Concerning the estimates of $\nabla^2T$ :
\begin{eqnarray*}
\Delta_f\arrowvert\nabla^2T\arrowvert^2&\geq&2\arrowvert\nabla^3T\arrowvert^2-c(n,\lambda)\arrowvert\nabla^2T\arrowvert^2(1+\arrowvert\Rm(g)\arrowvert)\\
&&-c(n,\lambda)\arrowvert\nabla^2T\arrowvert\left(\arrowvert\nabla T\arrowvert\arrowvert\nabla\Rm(g)\arrowvert+\arrowvert\nabla^2\Rm(g)\arrowvert\arrowvert T\arrowvert\right)\\
&\geq&2\arrowvert\nabla^3T\arrowvert^2-\psi_1\arrowvert\nabla^2T\arrowvert^2-\psi_2.
\end{eqnarray*}
where 
\begin{eqnarray*}
\psi_1&:=&c(n,\lambda)(1+\arrowvert\Rm(g)\arrowvert+\arrowvert\nabla\Rm(g)\arrowvert^{2/3}+\arrowvert\nabla^2\Rm(g)\arrowvert^{1/2})\\
\psi_2&:=&c(n,\lambda)(\arrowvert\nabla T\arrowvert^2\arrowvert\nabla\Rm(g)\arrowvert^{4/3}+\arrowvert\nabla^2\Rm(g)\arrowvert^{3/2}\arrowvert T\arrowvert^{2}).
\end{eqnarray*}

Now, define $U:=\arrowvert\nabla T\arrowvert^2$, $V:=\arrowvert\nabla^2T\arrowvert^2$ and $W:=\arrowvert\nabla^3T\arrowvert^2$. Consider the function $F:=(U+a)V$ where $a:=c\sup_{B(p,r)}U$. $F$ satisfies,
\begin{eqnarray*}
\Delta_f F\geq V^2-\left(2\psi_1+\frac{\arrowvert\nabla\Rm(g)\arrowvert^{4/3}\arrowvert T\arrowvert^2}{a}\right)F-2a\psi_2.
\end{eqnarray*}

Therefore, as in the proof of lemma \ref{elliptic}, we obtain :
\begin{eqnarray*}
\sup_{B(p,r/2)}V\leq c \sup_{B(p,r)}\left[a\psi_1+\arrowvert\nabla\Rm(g)\arrowvert^{4/3}\arrowvert T\arrowvert^2+(a\psi_2)^{1/2}\right]
\end{eqnarray*}
\end{proof}

\subsection{A priori estimates for sub solutions of the weighted laplacian}
The purpose of this section is to prove the following lemma that is crucial for the rest of the paper.

\begin{lemma}\label{lemm-cruc}
Let $(M^n,g,\nabla f)$ be a normalized expanding gradient Ricci soliton such that $f$ is an exhausting function. Assume $u:M^n\rightarrow$ is a $C^2$ function that is a sub solution of the following weighted elliptic equation :
\begin{eqnarray}
\Delta_fu\geq-\lambda u-c_1v^{-\alpha}u-Q,\label{subsol-ell-wei}
\end{eqnarray}
where $\lambda\in\mathbb{R}$, $c_1,\alpha\in\mathbb{R}_+^*$ and $Q:M^n\rightarrow \mathbb{R}$ is a nonnegative function. Define $\tilde{\alpha}:=\min\{1,\alpha\}.$ Assume $\sup_{M^n}f^{\tilde{\alpha}}\arrowvert \R_g\arrowvert$ is finite. \\
\begin{enumerate}
\item If $Q\equiv0$, then there exists some positive height $t_0=t_0(n,\lambda,\alpha,c_1,\min_{M^n}f,\sup_{M^n}f^{\tilde{\alpha}}\arrowvert\R_g\arrowvert)$ and some positive constants $C_i=C_i(n,\lambda,\alpha,c_1,\min_{M^n}f,\sup_{M^n}f^{\tilde{\alpha}}\arrowvert\R_g\arrowvert)$ for $i=0,1$  such that for $t\geq t_0$, the function
\begin{eqnarray*}
v^{\lambda}e^{-C_0v^{-\tilde{\alpha}}}u-A_1v^{2\lambda-\frac{n}{2}}e^{-v-C_1v^{-\tilde{\alpha}}}:\{t_0\leq f\leq t\}\rightarrow \mathbb{R},
\end{eqnarray*}
attains its maximum on the boundary $\{f=t_0\}\cup\{f=t\}$ for any nonnegative constant $A_1$.\\

\item If $Q=\textit{O}(v^{-\beta})$, for some real number $\beta$ such that $\lambda<\beta$ then there exists some positive height $t_0=t_0(n,\lambda,\alpha,c_1,\beta,\sup_{M^n}f^{\tilde{\alpha}}\arrowvert\R_g\arrowvert,\min_{M^n}f)$ and some positive constants $C_i=C_i(n,\lambda,\alpha,c_1,\min_{M^n}f,\sup_{M^n}f^{\tilde{\alpha}}\arrowvert\R_g\arrowvert)$ for $i=0,1$ such that for $t\geq t_0$, the function
\begin{eqnarray*}
v^{\lambda}e^{-C_0v^{-\tilde{\alpha}}}u-A_0v^{\lambda-\beta}-A_1v^{2\lambda-\frac{n}{2}}e^{-v-C_1v^{-\tilde{\alpha}}}:\{t_0\leq f\leq t\}\rightarrow \mathbb{R},
\end{eqnarray*}
attains its maximum on the boundary $\{f=t_0\}\cup\{f=t\}$ for any nonnegative constant $A_1$, and $A_0\geq A_0(n,\lambda,c_1,\alpha,\beta,\sup_{M^n}(v^{\beta}Q),\min_{M^n}f)$.\\

\item If $Q=\textit{O}(v^{\beta}e^{-v})$, for some real number $\beta$ such that $\lambda>\beta+n/2$, then there exists some positive height $t_0=t_0(n,\lambda,\alpha,c_1,\beta,\sup_{M^n}f^{\tilde{\alpha}}\arrowvert\R_g\arrowvert,\min_{M^n}f)$ and some positive constants $C_i=C_i(n,\lambda,\alpha,c_1,\min_{M^n}f,\sup_{M^n}f^{\tilde{\alpha}}\arrowvert\R_g\arrowvert)$ for $i=0,1$  such that for $t\geq t_0$, the function
\begin{eqnarray*}
v^{\lambda}e^{-C_0v^{-\tilde{\alpha}}}u+A_0v^{\beta+\lambda}e^{-v}-A_1v^{2\lambda-\frac{n}{2}}e^{-v-C_1v^{-\tilde{\alpha}}}:\{t_0\leq f\leq t\}\rightarrow \mathbb{R},
\end{eqnarray*}
attains its maximum on the boundary $\{f=t_0\}\cup\{f=t\}$ for any nonnegative constant $A_1$ and $A_0\geq A_0(n,\lambda,c_1,\alpha,\beta,\sup_{M^n}(v^{-\beta}e^vQ),\min_{M^n}f).$

\end{enumerate}
\end{lemma}

\begin{proof}
We first absorb the linear term on the right hand side of (\ref{subsol-ell-wei}) :
\begin{eqnarray*}
\Delta_f(v^{\lambda}u)\geq-\left(\lambda(\lambda+1)\arrowvert\nabla\ln v\arrowvert^2+c_1v^{-\alpha}\right)(v^{\lambda} u)-v^\lambda Q+2\lambda<\nabla\ln v,\nabla(v^{\lambda}u)>,
\end{eqnarray*}
i.e.
\begin{eqnarray*}
\Delta_{v-2\lambda\ln v}(v^{\lambda}u)&\geq&-\left(\lambda(\lambda+1)\arrowvert\nabla\ln v\arrowvert^2+c_1v^{-\alpha}\right)(v^{\lambda} u)-v^\lambda Q\\
&\geq&-\frac{C(\lambda,c_1,\min_{M^n}v)}{v^{\min\{\alpha,1\}}}(v^{\lambda}u)-v^\lambda Q.
\end{eqnarray*}
Then, if $\tilde{\alpha}:=\min\{\alpha,1\}$, multiply the previous differential inequality by a function of the form $e^{-Cv^{-\tilde{\alpha}}}$ to get,
\begin{eqnarray*}
\Delta_{v-2\lambda\ln v+2Cv^{-\tilde{\alpha}}}\left(v^{\lambda}e^{-Cv^{-\tilde{\alpha}}}u\right)&\geq& \Delta_{v-2\lambda\ln v}\left(e^{-Cv^{-\tilde{\alpha}}}\right)\cdot(v^{\lambda}u)-2C^2\arrowvert\nabla v^{-\tilde{\alpha}}\arrowvert^2(v^{\lambda}e^{-Cv^{-\tilde{\alpha}}}u)\\
&&-\frac{C(\lambda,c_1,\min_{M^n}v)}{v^{\tilde{\alpha}}}\left(v^{\lambda}e^{-Cv^{-\tilde{\alpha}}}u\right)-v^\lambda e^{-Cv^{-\tilde{\alpha}}} Q.
\end{eqnarray*}
Now,
\begin{eqnarray*}
 \Delta_{v-2\lambda\ln v}\left(e^{-Cv^{-\tilde{\alpha}}}\right)&=&\left(-C\Delta_{v-2\lambda\ln v}v^{-\tilde{\alpha}}+C^2\arrowvert\nabla v^{-\tilde{\alpha}}\arrowvert^2\right)e^{-Cv^{-\tilde{\alpha}}}\\
 &=&\frac{C}{v^{\tilde{\alpha}}}\left(\tilde{\alpha}(1+[\tilde{\alpha}+1-2\lambda+C\tilde{\alpha}v^{-\tilde{\alpha}}]\arrowvert\nabla\ln v\arrowvert^2)\right)e^{-Cv^{-\tilde{\alpha}}}\\
 &\geq&\frac{C\tilde{\alpha}}{2v^{\tilde{\alpha}}}e^{-Cv^{-\tilde{\alpha}}},
 \end{eqnarray*}
for $v\geq v_0(C,\min_{M^n}v,\alpha).$
Therefore,
\begin{eqnarray}
\Delta_{v-2\lambda\ln v+2Cv^{-\tilde{\alpha}}}\left(v^{\lambda}e^{-Cv^{-\tilde{\alpha}}}u\right)&\geq&\left(\frac{C\tilde{\alpha}}{2}-C(\lambda,c_1,\min_{M^n}v)\right)v^{-\tilde{\alpha}}\left(v^{\lambda}e^{-Cv^{-\tilde{\alpha}}}u\right)\\
&&-2C^2\arrowvert\nabla v^{-\tilde{\alpha}}\arrowvert^2(v^{\lambda}e^{-Cv^{-\tilde{\alpha}}}u)-v^\lambda e^{-Cv^{-\tilde{\alpha}}} Q,\\
\Delta_{v-2\lambda\ln v+2Cv^{-\tilde{\alpha}}}\left(v^{\lambda}e^{-Cv^{-\tilde{\alpha}}}u\right)&>&-v^\lambda e^{-Cv^{-\tilde{\alpha}}} Q,\label{Generic-Q}
\end{eqnarray}
for $C=C(\lambda,c_1,\alpha,\min_{M^n}v)$ and $v\geq v_0(\lambda,c_1,\min_{M^n}v,\alpha).$
\begin{enumerate}
\item If $Q\equiv0$, then
\begin{eqnarray}
\Delta_{v-2\lambda\ln v+2Cv^{-\tilde{\alpha}}}\left(v^{\lambda}e^{-Cv^{-\tilde{\alpha}}}u\right)&>&0,\label{Q-vanish}
\end{eqnarray}
for $C=C(\lambda,c_1,\alpha,\min_{M^n}v)$ and $v\geq v_0(\lambda,c_1,\min_{M^n}v,\alpha).$
Now, we compute the weighted laplacian of a function of the form $v^{\gamma}e^{-v}$ with $\gamma\in\mathbb{R}$ :

\begin{eqnarray*}
\Delta_{v-2\lambda \ln v}v^{\gamma}&=&\left(\gamma v^{\gamma -1}\Delta_{v-2\lambda \ln v}v+\gamma(\gamma-1)\arrowvert\nabla\ln v\arrowvert^2v^{\gamma}\right)\\
&=&\gamma v^{\gamma}\left(1+(\gamma-1-2\lambda)\arrowvert\nabla\ln v\arrowvert^2\right),\\
\Delta_{v-2\lambda \ln v}e^{-v}&=&(-\Delta_{v-2\lambda \ln v} v +\arrowvert\nabla v\arrowvert^2)e^{-v}=\left(-\Delta v+2\lambda\frac{\arrowvert\nabla v\arrowvert^2}{v}\right)e^{-v}\\
&=&\left(-\R_g-\lambda\frac{2\R_g+n}{v}+2\lambda-\frac{n}{2}\right)e^{-v},\\
\Delta_{v-2\lambda \ln v}(v^{\gamma}e^{-v})&=&\gamma v^{\gamma}e^{-v}\left(1+(\gamma-1-2\lambda)\arrowvert\nabla\ln v\arrowvert^2\right)-2\gamma v^{\gamma}e^{-v}\frac{\arrowvert\nabla v\arrowvert^2}{v}\\
&&+\left(-\R_g-\lambda\frac{2\R_g+n}{v}+2\lambda-\frac{n}{2}\right)v^{\gamma}e^{-v}\\
&=&\left(2\lambda-\frac{n}{2}-\gamma-\R_g+(\gamma-\lambda)\frac{2\R_g+n}{v}+\gamma(\gamma-1-2\lambda)\arrowvert\nabla\ln v\arrowvert^2\right)v^{\gamma}e^{-v}.\\
<\nabla v^{-\tilde{\alpha}},\nabla v^{\gamma}e^{-v}>&\leq&\frac{C(\gamma,\min_{M^n}v)}{v^{\tilde{\alpha}}}v^{\gamma}e^{-v}.
\end{eqnarray*}
Therefore, if $\gamma:=2\lambda-n/2$, then,
\begin{eqnarray}
\Delta_{v-2\lambda \ln v+2Cv^{-\tilde{\alpha}}}(v^{2\lambda-n/2}e^{-v})\leq \frac{C_1}{v^{\tilde{\alpha}}}v^{2\lambda-n/2}e^{-v},\label{spe-fct-sub-ell}
\end{eqnarray}
for some positive constant $C_1=C_1(C,\sup_{M^n}v^{\tilde{\alpha}}\arrowvert\R_g\arrowvert,n,\min_{M^n}f)$. Now, 
\begin{eqnarray*}
\Delta_{v-2\lambda \ln v+2Cv^{-\tilde{\alpha}}}(v^{2\lambda-n/2}e^{-v}e^{-2C_2v^{-\tilde{\alpha}}})\leq\left(\frac{C_1-C_2}{v^{\tilde{\alpha}}}\right)v^{2\lambda-n/2}e^{-v}e^{-2C_2v^{-\tilde{\alpha}}},
\end{eqnarray*}
for any positive constant $C_2$, for $v\geq v_0(C_2,C,n,\min_{M^n}v)$ since,
\begin{eqnarray*}
\Delta_{v-2\lambda \ln v+2Cv^{-\tilde{\alpha}}}\left(e^{-2C_2v^{-\tilde{\alpha}}}\right)&\leq&\frac{C_2\tilde{\alpha}}{v^{\tilde{\alpha}}}e^{-2C_2v^{-\tilde{\alpha}}},\\
2<\nabla (v^{2\lambda-n/2}e^{-v}),\nabla e^{-2C_2v^{-\tilde{\alpha}}}>&\leq&-\frac{2C_2\tilde{\alpha}}{v^{\tilde{\alpha}}}v^{2\lambda-n/2}e^{-v}e^{-2C_2v^{-\tilde{\alpha}}}.
\end{eqnarray*}

Choose $C_2=C_1/{\tilde{\alpha}}$ such that  
\begin{eqnarray}
\Delta_{v-2\lambda \ln v+2Cv^{-\tilde{\alpha}}}(v^{2\lambda-n/2}e^{-v}e^{-2C_1v^{-\tilde{\alpha}}})\leq 0,\label{spe-fct}
\end{eqnarray}
for $v\geq v_0(C,\sup_{M^n}v^{\tilde{\alpha}}\arrowvert\R_g\arrowvert)$.\\
Combining the two differential inequalities (\ref{Q-vanish}) and (\ref{spe-fct}), one gets for any nonnegative constant $A_1$, for $v\geq v_0(n,\lambda,c_1,\min_{M^n}v,\alpha,\sup_{M^n}v^{\tilde{\alpha}}\arrowvert\R_g\arrowvert)$,
\begin{eqnarray}
\Delta_{v-2\lambda \ln v+2Cv^{-\tilde{\alpha}}}\left(v^{\lambda}e^{-Cv^{-\tilde{\alpha}}}u-A_1v^{2\lambda-n/2}e^{-v}e^{-C_1v^{-\tilde{\alpha}}}\right)> 0,\label{Q-vanish-bis}
\end{eqnarray}
where $C=C(n,\lambda,c_1,\alpha,\min_{M^n}v)$ and $C_1=C_1(n,\lambda,c_1,\alpha,\min_{M^n}v,\sup_{M^n}v^{\tilde{\alpha}}\arrowvert\R_g\arrowvert)$.
Hence the result by applying the maximum principle on a slice $\{v_0\leq v\leq t\}$.\\

\item If $Q=\textit{O}(v^{-\beta})$ with $\lambda<\beta$, then, thanks to (\ref{Generic-Q}) and (\ref{Q-vanish-bis}),
\begin{eqnarray*}
\Delta_{v-2\lambda\ln v+2Cv^{-\tilde{\alpha}}}\left(v^{\lambda}e^{-Cv^{-\tilde{\alpha}}}u-A_1v^{2-n/2}e^{-v}e^{-C_1v^{-\tilde{\alpha}}}\right)&\geq&-\tilde{Q}v^{\lambda-\beta} e^{-Cv^{-\tilde{\alpha}}},\\
\end{eqnarray*}
where $\tilde{Q}:=\sup_{M^n}v^{\beta}Q,$ for $v\geq v_0(n,\lambda,c_1,\alpha,\sup_{M^n}(v^{\tilde{\alpha}}\arrowvert\R_g\arrowvert),\min_{M^n}v)$.
 We notice that, for $\gamma>0$,
\begin{eqnarray*}
\Delta_{v-2\lambda\ln v+2Cv^{-\tilde{\alpha}}}\left(v^{-\gamma}\right)\leq -\frac{\gamma}{2}v^{-\gamma},
\end{eqnarray*}
for $v\geq v_0(n,\lambda,C,\alpha,\gamma,\min_{M^n}v)$. Therefore, if $\gamma:=\beta-\lambda$,
\begin{eqnarray*}
\Delta_{v-2\lambda\ln v+2Cv^{-\tilde{\alpha}}}\left(v^{\lambda}e^{-Cv^{-\tilde{\alpha}}}u-A_1v^{2\lambda-n/2}e^{-v}e^{-C_1v^{-\tilde{\alpha}}}-A_0v^{\lambda-\beta}\right)>0,
\end{eqnarray*}
for 
\begin{eqnarray*}
&&v\geq v_0(n,\lambda,c_1,\alpha,\beta,\sup_{M^n}(v^{\tilde{\alpha}}\arrowvert\R_g\arrowvert),\min_{M^n}v),\\
&&A_0\geq A_0(n,\lambda,c_1,\alpha,\beta,\sup_{M^n}(v^{\beta}Q),\min_{M^n}v),\\
&&A_1\geq0.
\end{eqnarray*}
Again, applying the maximum principle on a slice $\{v_0\leq v\leq t\}$ gives the result.\\
\item If $Q=\textit{O}(v^{\beta}e^{-v})$ with $\beta+n/2<\lambda$, then,
\begin{eqnarray*}
\Delta_{v-2\lambda\ln v+2Cv^{-\tilde{\alpha}}}\left(v^{\lambda}e^{-Cv^{-\tilde{\alpha}}}u-A_1v^{2\lambda-n/2}e^{-v}e^{-C_1v^{-\tilde{\alpha}}}\right)&\geq&-\tilde{Q}v^{\lambda+\beta} e^{-v-Cv^{-\tilde{\alpha}}},\\
\end{eqnarray*}
 where $\tilde{Q}:=\sup_{M^n}v^{-\beta}e^vQ,$ for $v\geq v_0(n,\lambda,c_1,\alpha,\sup_{M^n}(v^{\tilde{\alpha}}\arrowvert\R_g\arrowvert),\min_{M^n}v)$.
A similar computation leading to the estimate (\ref{spe-fct-sub-ell}) gives
\begin{eqnarray*}
\Delta_{v-2\lambda\ln v+2Cv^{-\tilde{\alpha}}}\left(v^{\lambda+\beta}e^{-v}\right)\geq \frac{\lambda-\beta-n/2}{2}v^{\lambda+\beta}e^{-v},
\end{eqnarray*}
for $v\geq v_0(n,\lambda,C,\alpha,\lambda,\beta,\sup_{M^n}(v^{\tilde{\alpha}}\arrowvert\R_g\arrowvert),\min_{M^n}v)$. Therefore,
\begin{eqnarray*}
\Delta_{v-2\lambda\ln v+2Cv^{-\tilde{\alpha}}}\left(v^{\lambda}e^{-Cv^{-\tilde{\alpha}}}u-A_1v^{2\lambda-n/2}e^{-v}e^{-C_1v^{-\tilde{\alpha}}}+A_0v^{\lambda+\beta}e^{-v}\right)\geq 0,
\end{eqnarray*}
for 
\begin{eqnarray*}
&&v\geq v_0(n,\lambda,c_1,\alpha,\beta,\sup_{M^n}(v^{\tilde{\alpha}}\arrowvert\R_g\arrowvert),\min_{M^n}v),\\
&&A_0\geq A_0(n,\lambda,c_1,\alpha,\beta,\sup_{M^n}(v^{-\beta}e^vQ),\min_{M^n}v),\\
&&A_1\geq0.
\end{eqnarray*}

\end{enumerate}
\end{proof}

The main application of lemma \ref{lemm-cruc} is the following corollary dealing with tensors :

\begin{coro}\label{coro-cruc-tens}
Let $(M^n,g,\nabla f)$ be a normalized expanding gradient Ricci soliton such that $f$ is an exhausting function. Assume there exists a $C^2$ tensor $T$ such that
\begin{eqnarray*}
\Delta_fT=-\lambda T+\Rm(g)\ast T,
\end{eqnarray*}
for some real number $\lambda$. Assume $R_0:=\sup_{M^n}\left(v^{\min\{\alpha,1\}}\arrowvert\Rm(g)\arrowvert\right)$ is finite, for some $\alpha>0$. 
\begin{enumerate}
\item Then,
\begin{eqnarray*}
&&\sup_{M^n}v^{\lambda}\arrowvert T\arrowvert\leq  C\left(\limsup_{+\infty}v^{\lambda}\arrowvert T\arrowvert+\sup_{v\leq v_0}\arrowvert T\arrowvert\right),\\
&&C=C(n,\lambda,\alpha,R_0,\min_{M^n}f),\\
&&v_0=v_0(n,\lambda,\alpha,R_0,\min_{M^n}f).
\end{eqnarray*}

Moreover, if $\limsup_{+\infty}v^{\lambda}\arrowvert T\arrowvert=0$ then
\begin{eqnarray*}
&&\sup_{M^n}v^{\frac{n}{2}-\lambda}e^v\arrowvert T\arrowvert\leq C\left(n,\lambda,\alpha,R_0,\min_{M^n}f,\sup_{v\leq v_0}\arrowvert T\arrowvert\right),\\
&&v_0=v_0(n,\lambda,\alpha,R_0,\min_{M^n}f).\\
\end{eqnarray*}

\item If $\sup_{M^n}v^{\min\{\alpha,1\}+i/2}\arrowvert\nabla^i\Rm(g)\arrowvert=:R_i^{\alpha}$ and $\limsup_{+\infty}v^{\lambda+i/2}\arrowvert\nabla^iT\arrowvert=:T_i$ are finite for $i=0,...,k$ then
\begin{eqnarray*}
\sup_{M^n}v^{\lambda+k/2}\arrowvert\nabla^kT\arrowvert\leq C\left(n,k,\lambda,\alpha,(R_i^{\alpha})_{0\leq i\leq k},\sup_{v\leq v_0}\arrowvert T\arrowvert,(T_i)_{0\leq i\leq k},\min_{M^n}f\right),
\end{eqnarray*}
where 
\begin{eqnarray*}
v_0=v_0\left(n,k,\lambda,\alpha,(R_i^{\alpha})_{0\leq i\leq k},\min_{M^n}f\right).\\
\end{eqnarray*}
\item If $T_0=0$ and if $R_0^{\alpha}$ is finite then
\begin{eqnarray*}
&&\sup_{M^n}v^{-\lambda-\frac{k}{2}+\frac{n}{2}}e^{v}\arrowvert\nabla^kT\arrowvert\leq C\left(n,k,\alpha,\lambda,R_0^{\alpha},\sup_{v\leq v_0}\arrowvert T\arrowvert,\min_{M^n}f\right),\\
&&v_0=v_0\left(n,k,\lambda,\alpha,R_0^{\alpha},\min_{M^n}f\right).
\end{eqnarray*}
\end{enumerate}
\end{coro}

\begin{proof}
\begin{enumerate}
\item We want to apply lemma \ref{lemm-cruc} to the function $u_{\epsilon}:=\sqrt{\arrowvert T\arrowvert^2+\epsilon^2}$. Indeed, $u_{\epsilon}$ satisfies
\begin{eqnarray*}
\Delta_fu_{\epsilon}\geq-\lambda u_{\epsilon}-c_1v^{-\alpha}u_{\epsilon}.
\end{eqnarray*}
Therefore, there exists some positive height $v_0=v_0(n,\lambda,\alpha,R_0^{\alpha},\min_{M^n}f)$ and some positive constants $C_i=C_i(n,\lambda,\alpha,R_0^{\alpha},\min_{M^n}f)$ for $i=0,1$  such that for $t\geq v_0$, the function
\begin{eqnarray*}
v^{\lambda}e^{-C_0v^{-\tilde{\alpha}}}u_{\epsilon}-A_1v^{2\lambda-\frac{n}{2}}e^{-v-C_1v^{-\tilde{\alpha}}}:\{v_0\leq f\leq t\}\rightarrow \mathbb{R},
\end{eqnarray*}
attains its maximum on the boundary $\{v=v_0\}\cup\{v=t\}$ for any nonnegative constant $A_1$. If we choose $A_1$ large enough (independent of $\epsilon\leq 1$) such that
\begin{eqnarray*}
&&\sup_{v\leq v_0}\left(v^{\lambda}e^{-C_0v^{-\tilde{\alpha}}}u_{\epsilon}-A_1v^{2\lambda-\frac{n}{2}}e^{-v-C_1v^{-\tilde{\alpha}}}\right)\leq 0,\\
&&A_1=A_1(n,\lambda,\alpha,R_0^{\alpha},\min_{M^n}f,\sup_{v\leq v_0}u).
\end{eqnarray*}

Hence the result by letting $\epsilon$ go to $0$ and then letting $v=t$ go to $+\infty$.\\

\item 
We are now in a position to prove the second part of corollary \ref{coro-cruc-tens}. Indeed, by using (\ref{comm-1}) and (\ref{comm-2}) of proposition \ref{id-op}, for $k\geq 1$,
\begin{eqnarray*}
\Delta_f(\nabla^kT)&=&\nabla^k\Delta_fT+[\Delta_f,\nabla^k]T\\
&=&-\left(\lambda+\frac{k}{2}\right)\nabla^kT+\Rm(g)\ast\nabla^kT+\sum_{i=0}^{k-1}\nabla^{k-i}\Rm(g)\ast\nabla^i T\\
&=&-\left(\lambda+\frac{k}{2}\right)\nabla^kT+\Rm(g)\ast\nabla^kT+Q_k,
\end{eqnarray*}
where $Q_k:=\sum_{i=0}^{k-1}\nabla^{k-i}\Rm(g)\ast\nabla^i T$. By induction on $k$, $k\geq2$ (the proof of case $k=1$ is straightforward )
\begin{eqnarray*}
&&\sup_{M^n}v^{\lambda+\frac{k}{2}+\tilde{\alpha}}\arrowvert Q_k\arrowvert =C\left(n,k,\lambda,\alpha,(R_i^{\alpha})_{0\leq i\leq k},\sup_{v\leq v_{k-1}}\arrowvert T\arrowvert,(T_i)_{0\leq i\leq k-1},\min_{M^n}f\right),\\
&&v_{k-1}=v_{k-1}\left(n,k,\lambda,\alpha,(R_i^{\alpha})_{0\leq i\leq k-1},\min_{M^n}f\right).
\end{eqnarray*}
By lemma \ref{lemm-cruc}, as $\lambda+k/2<\lambda+k/2+\tilde{\alpha}$, there exists some positive height $$v_k=v_k(n,k,\lambda,\alpha,R_0^{\alpha}),\min_{M^n}f)$$ and some positive constants $C_i=C_i(n,\lambda,\alpha,R_0^{\alpha},\min_{M^n}f)$ for $i=0,1$ such that for $t\geq v_k$, the function
\begin{eqnarray*}
v^{\lambda+\frac{k}{2}}e^{-C_0v^{-\tilde{\alpha}}}\arrowvert\nabla^kT\arrowvert-A_0v^{-\frac{k}{2}-\tilde{\alpha}}-A_1v^{2\lambda+k-\frac{n}{2}}e^{-v-C_1v^{-\tilde{\alpha}}}:\{v_k\leq v\leq t\}\rightarrow \mathbb{R},
\end{eqnarray*}
attains its maximum on the boundary $\{v=v_k\}\cup\{v=t\}$ for any nonnegative constant $A_1$, and $A_0\geq A_0(n,k,\lambda,R_0^{\alpha},\alpha,\sup_{M^n}(v^{\lambda+\frac{k}{2}+\tilde{\alpha}}\arrowvert Q_k\arrowvert),\min_{M^n}f)$. For $A_0$ large enough, one can ensure that
\begin{eqnarray*}
&&\sup_{v=v_k}\left(v^{\lambda+\frac{k}{2}}e^{-C_0v^{-\tilde{\alpha}}}\arrowvert\nabla^kT\arrowvert-A_0v^{-\frac{k}{2}-\tilde{\alpha}}\right)\leq 0,\\
&& A_0\geq C(n,k,\lambda,R_0^{\alpha},\alpha,\sup_{M^n}(v^{\lambda+\frac{k}{2}+\tilde{\alpha}}\arrowvert Q_k\arrowvert),\min_{M^n}f)\sup_{v\leq v_k}\arrowvert\nabla^kT\arrowvert.
\end{eqnarray*}
Therefore, if $t$ goes to $+\infty$, 
\begin{eqnarray*}
&&\sup_{M^n}v^{\lambda+\frac{k}{2}+\tilde{\alpha}}\arrowvert\nabla^kT\arrowvert\leq C(\limsup_{+\infty}v^{\lambda+\frac{k}{2}+\tilde{\alpha}}\arrowvert\nabla^kT\arrowvert+\sup_{v\leq v_k}\arrowvert\nabla^k T\arrowvert),\\
&&C=C\left(n,k,\lambda,\alpha,(R_i^{\alpha})_{0\leq i\leq k},\sup_{v\leq v_{k-1}}\arrowvert T\arrowvert,(T_i)_{0\leq i\leq k},\sup_{v\leq v_{k}}\arrowvert \nabla^kT\arrowvert,\min_{M^n}f\right),\\
&&v_k=v_k\left(n,k,\lambda,\alpha,(R_i^{\alpha})_{0\leq i\leq k},\min_{M^n}f\right).
\end{eqnarray*}
Now, by the local estimates given by lemma \ref{loc-est-tensor-egs},
\begin{eqnarray*}
\sup_{v\leq v_k}\arrowvert\nabla^k T\arrowvert&\leq& C(n,k,\lambda,\alpha,\sup_{v\leq 2v_k}\arrowvert\Rm(g)\arrowvert,\min_{M^n}f) \sup_{v\leq 2v_k}\arrowvert T\arrowvert.\\
&\leq& C(n,k,\lambda,\alpha,R_0^{\alpha},\min_{M^n}f) \sup_{v\leq 2v_k}\arrowvert T\arrowvert.
\end{eqnarray*}
which concludes the proof.\\

\item The proof of $(3)$ is similar. Here, by induction on $k$,
\begin{eqnarray*}
&&\sup_{M^n}\left(v^{\frac{n}{2}-\lambda-\frac{k-1}{2}}e^v\arrowvert Q_k\arrowvert\right)\leq C\left(n,k,\lambda,\alpha,\sup_{v\leq v_{k-1}}\arrowvert T\arrowvert,R_0^{\alpha},\min_{M^n}f\right),\\
&&v_{k-1}=v_{k-1}\left(n,k,\lambda,\alpha,R_0^{\alpha},\min_{M^n}f\right).
\end{eqnarray*}
According to lemma \ref{lemm-cruc}, since trivially, $$\lambda+\frac{k}{2}>\lambda+\frac{k-1}{2},$$
there exists some positive height $$v_k=v_k(n,k,\lambda,\alpha,R_0^{\alpha},\min_{M^n}f)$$ and some positive constants $C_i=C_i(n,\lambda,\alpha,R_0^{\alpha},\min_{M^n}f)$ for $i=0,1$  such that for $t\geq v_k$, the function defined on $\{v_k\leq v\leq t\}$,
\begin{eqnarray*}
v^{\lambda+\frac{k}{2}}e^{-C_0v^{-\tilde{\alpha}}}\arrowvert\nabla^kT\arrowvert+A_0v^{2\lambda+k-\frac{n+1}{2}}e^{-v}-A_1v^{2\lambda+k-\frac{n}{2}}e^{-v-C_1v^{-\tilde{\alpha}}},
\end{eqnarray*}
attains its maximum on the boundary $\{v=v_k\}\cup\{v=t\}$ for any nonnegative constant $A_1$ and $A_0\geq A_0(n,\lambda,R_0^{\alpha},\alpha,\sup_{M^n}(v^{\frac{n}{2}-\lambda-\frac{k-1}{2}}e^v\arrowvert Q_k\arrowvert),\min_{M^n}f).$ For $A_1$ large enough, one can ensure that
\begin{eqnarray*}
&&\sup_{v=v_k}\left(v^{\lambda+\frac{k}{2}}e^{-C_0v^{-\tilde{\alpha}}}\arrowvert\nabla^kT\arrowvert+A_0v^{2\lambda+\frac{k-1}{2}}e^{-v}-A_1v^{2\lambda+k-\frac{n}{2}}e^{-v-C_1v^{-\tilde{\alpha}}}\right)\leq 0,\\
&& A_1\geq C(n,k,\lambda,R_0^{\alpha},\alpha,\sup_{M^n}(v^{\frac{n}{2}-\lambda-\frac{k-1}{2}}e^v\arrowvert Q_k\arrowvert),\min_{M^n}f)\sup_{v\leq v_k}\arrowvert\nabla^kT\arrowvert.
\end{eqnarray*}
Now, by the local estimates given by lemma \ref{loc-est-tensor-egs}, as $T=\textit{O}(v^{2\lambda-n/2}e^{-v})$, $T_k=0$ for any $k\geq 0$.

Therefore, if $t$ goes to $+\infty$ together with lemma \ref{loc-est-tensor-egs},
\begin{eqnarray*}
&&\sup_{M^n}\left(v^{\frac{n}{2}-\lambda-\frac{k}{2}}e^v\arrowvert\nabla^kT\arrowvert\right)\leq C(n,k,\alpha,\lambda,\min_{M^n}f,R_0^{\alpha},\sup_{v\leq v_k}\arrowvert T\arrowvert),\\
&&v_k=v_k\left(n,k,\lambda,\alpha,R_0^{\alpha},\min_{M^n}f\right).
\end{eqnarray*}

\end{enumerate}
\end{proof}

\section{Asymptotic estimates}\label{asy-est-sec}
\subsection{ Rate of the convergence to the asymptotic cone : generic case}

By the very definition of the asymptotic cone, the asymptotic geometry is close to the cone $(C(X),dr^2+r^2g_{X})$ in the sense of smooth Cheeger-Gromov convergence. The aim of this section is to quantify this convergence. This approach is inspired by \cite{Ban-Kas-Nak} and \cite{Che-Tian-Ric-Fla} which are the cornerstones of the study of asymptotically locally euclidean manifolds (ALE) : see also \cite{Ach-Via}.

\begin{defn}
  $\A_g^k(T):=\limsup_{x\rightarrow+\infty}r_p(x)^{2+k}\arrowvert \nabla^kT\arrowvert_g(x)$ for a tensor $T$ and a nonnegative integer $k$.
\end{defn}
The purpose of this section is to give a proof of the following theorem.

\begin{theo}\label{asy-con-gen}
Let $(M,g,\nabla f)$ be an expanding gradient Ricci soliton such that
\begin{eqnarray*}
 \A^k_g(\Ric(g))<+\infty,\quad \forall k\in\mathbb{N}.
\end{eqnarray*}

Then there exists a metric cone over a smooth compact manifold $(C(X),dr^2+r^2g_X)$ with a smooth metric such that $(M,g,\nabla f)$ is asymptotic to $(C(X),dr^2+r^2g_X,r\partial_r/2)$ at polynomial rate $\tau=2$. In particular, the invariants $(\A_g^k(\Rm(g))_{k\geq 0}$ are finite.

\end{theo}

\begin{proof}
The proof is essentially the same as in \cite{Che-Der}. Nonetheless, we clarify the situation by pointing out the crucial estimates.
 Indeed, as the Ricci curvature has quadratic curvature decay, one can show that the potential function is equivalent to $r_p^2/4$ for any $p\in M$, hence proper and by (\ref{equ:4}), the level sets of $f+\mu(g)$ (this choice will be made clearer below), denoted by $M_t:=\{f+\mu(g)=t\}$, are well-defined compact diffeomorphic hypersurfaces, for $t$ large enough. Let's define the following diffeomorphism \begin{eqnarray*}
\phi:[t_0,+\infty)\times M_{t_0^2/4}=:[t_0,+\infty)\times X &\rightarrow& M_{\geq t_0^2/4}\\
(x,t)&\rightarrow& \phi_{\frac{t^2}{4}-\frac{t_0^2}{4}}(x),
\end{eqnarray*}
for $t_0$ large enough, where $(\phi_t)_t$ is the Morse flow associated to the potential function $f+\mu(g)$, i.e. $$\partial_t\phi_t=\frac{\nabla f}{\arrowvert\nabla f\arrowvert^2}.$$
Then the pull-back metric $\phi^*g$ can be written as follows
\begin{eqnarray*}
\phi^*g&=&\frac{t^2}{4\arrowvert\nabla f\arrowvert^2}dt^2+t^2\bar{g}_t\\
&=&\left(1+\frac{4\phi^*\R_g}{t^2-4\phi^*\R_g}\right)dt^2+t^2\bar{g}_t,
\end{eqnarray*}
where, by definition, $\bar{g}_t:=t^{-2}\phi_{t^2/4-t_0^2/4}^*(g|_{M_{t^2/4}}).$
Therefore, showing the expected asymptotic estimates amounts to study the convergence of the one parameter family of metrics $\bar{g}_t$ as $t$ goes to $+\infty$. Moreover, by the very definition of $\phi$, one checks immediately that it is compatible with the soliton structure :
\begin{eqnarray*}
(f+\mu(g))(\phi(t,x))=\frac{t^2}{4}\quad;\quad\partial_t\phi=\sqrt{f+\mu(g)}\frac{\nabla f}{\arrowvert\nabla f\arrowvert^2}.
\end{eqnarray*}

Under the sole assumption $\A_g^0(\Rm(g))<+\infty$, Chen and the author show in \cite{Che-Der} that $(\bar{g}_t)_{t\geq t_0}$ $C^{1,\alpha}$ converges to a $C^{1,\alpha}$ metric $\bar{g}_{\infty}$ without giving the rate of the convergence. Moreover, if one imposes bounds on the rescaled covariant derivatives of the Ricci tensor, the family $(\bar{g}_t)_t$ converges to $\bar{g}_{\infty}$ in the $C^{\infty}$ topology.

Indeed, we are even able to precise the convergence of this family of metrics to its limit $\bar{g}_{\infty}$ : 
\begin{eqnarray}
\partial_t\bar{g}_t&=&-2t^{-1}\bar{g}_t+t^{-2}\cdot\frac{t}{2}\frac{2\phi^*\nabla^2f}{\arrowvert\nabla f\arrowvert^2}\\
&=&\left(-2t^{-1}+\frac{t}{2\arrowvert\nabla f\arrowvert^2}\right)\bar{g}_t+\frac{\phi^*\Ric(g)}{t\arrowvert\nabla f\arrowvert^2}\\
&=&\frac{8\phi^*\R_g}{t(t^2-\phi^*\R_g)}\bar{g}_t+\frac{4\phi^*\Ric(g)}{t(t^2-4\phi^*\R_g)}. \label{est-rate-con-exp}
\end{eqnarray}
Hence, if the Ricci curvature decays quadratically,
\begin{eqnarray}
\bar{g}_t-\bar{g}_{\infty}=\textit{O}(t^{-2}).\label{dec-c^0-con-exp}
\end{eqnarray}

With appropriate bounds on the rescaled covariant derivatives of the Ricci curvature, one gets more precise estimates concerning the covariant derivatives of the difference  $\bar{g}_t-\bar{g}_{\infty}$ : firstly, one derivates equation (\ref{est-rate-con-exp}) and then, by using the assumption on the decay of the covariant derivatives of the Ricci curvature, one integrates as we did previously to get $\nabla^{\bar{g}_{\infty},i}\bar{g}_t=\textit{O}(t^{-2})$ for $i\in\mathbb{N}$.

We finally prove that the full curvature tensor decays quadratically if a sufficient number (two exactly) of covariant derivatives of the Ricci curvature are bounded in a rescaled sense.

By a previous observation, if $\A^0_g(\Ric(g))<+\infty$, $f$ is quadratic in the distance to a fixed point and its Morse flow is well-defined outside a compact set. Therefore, by proposition \ref{id-op}, one obtains on $M_t$ :
\begin{eqnarray*}
\partial_t\arrowvert\Rm(g)\arrowvert^2+\frac{4\nabla^2 f(\Rm(g),\Rm(g))}{\arrowvert\nabla f\arrowvert^2}=\frac{1}{\arrowvert\nabla f\arrowvert^2}\nabla^2\Ric(g)\ast\Rm(g),
\end{eqnarray*}
which means in particular that
\begin{eqnarray}
\partial_t(t^2\arrowvert\Rm(g)\arrowvert^2)\leq\textit{O}(t^{-2})(t^2\arrowvert\Rm(g)\arrowvert^2)+c(n)\arrowvert\nabla^2\Ric(g)\arrowvert(t\arrowvert\Rm(g)\arrowvert).\label{bd-curv-op-gen-case}
\end{eqnarray}
By integrating this differential inequality, we get for $t_0\leq t$,
\begin{eqnarray*}
t\sup_{M_t}\arrowvert\Rm(g)\arrowvert\leq t_0\sup_{M_{t_0}}\arrowvert\Rm(g)\arrowvert+\textit{O}\left(\int_{t_0}^t\sup_{M_s}\arrowvert\nabla^2\Ric(g)\arrowvert ds\right),
\end{eqnarray*}
giving the finiteness of $\A^0_g(\Rm(g)$ if $\A_g^2(\Ric(g))<+\infty$. The finiteness of the invariants $\A^k_g(\Rm(g))$ follows by induction by mimicking the previous argument together with the second part of proposition \ref{id-op}.

\end{proof}

\begin{rk}
Theorem \ref{asy-con-gen} does not assume a priori bounds on the full curvature tensor : the sole assumption on the rescaled covariant derivatives of the Ricci curvature implies the corresponding bounds for the curvature tensor. Nonetheless, the proof of theorem \ref{asy-con-gen} shows a loss of two derivatives when estimating the convergence and the finiteness of the invariants $(\A_g^k(\Rm(g))_{k\geq 0}$ which seems to be a usual phenomenon when dealing with geodesic normal coordinates : \cite{Che-Tian-Ric-Fla} and \cite{Ban-Kas-Nak}. Can one prove the existence of a harmonic diffeomorphism at infinity which preserves the soliton structure ?
\end{rk}
\subsection{Ricci flat asymptotic cones and rigidity of pinched metrics}

The first part of this section consists in proving the following theorem :

\begin{theo}\label{dec-Ricci-flat-cone}
Let $(M^n,g,\nabla f)$ be an expanding gradient Ricci soliton. Assume $$\lim_{+\infty}r_p^2\Ric(g)=0.$$
\begin{enumerate}
\item Then, $\nabla^k\Ric(g)=\textit{O}(f^{1+\frac{k-n}{2}}e^{-f}),$ for any $k\geq 0$.\\

\item $(M,g,\nabla f)$ is asymptotic to a smooth Ricci flat metric cone $(C(X),dr^2+r^2g_X)$ at exponential rate. \\

\item If the scalar curvature is positive then 
\begin{eqnarray*}
\inf_M f^{n/2-1}e^f\R_g\geq C(n,\sup_{M^n}\R_g,\sup_{M^n} f\R_g)\min\{\liminf_{+\infty}f^{n/2-1}e^f\R_g ; \min_{f\leq C(\mu(g),n,\sup_{M^n}\R_g)}\R_g\}.\\
\end{eqnarray*}
\end{enumerate}
\end{theo}

\begin{rk}
\begin{itemize}
\item The assumption on the Ricci curvature is satisfied if, for instance, the convergence to its asymptotic cone is at least $C^k$, for $k\geq 2$ and if the asymptotic cone is Ricci flat.

\item The Ricci curvature decay obtained in theorem \ref{dec-Ricci-flat-cone} is sharp. Indeed, according to \cite{Sie-PHD}, the Ricci curvature decay of the Ricci expanders $(L^{-k},g_{k,1})_{k}$ with $k\in\mathbb{N}^*$ coming out of the cone $(C(\mathbb{S}^{2n-1}/\mathbb{Z}_k),i\partial\bar{\partial}(\arrowvert z\arrowvert^{2}))$ built in \cite{Fel-Ilm-Kno} is exactly $\textit{O}(f^{1-n/2}e^{-f})$. 
\end{itemize}
\end{rk}

\begin{proof}
\begin{enumerate}
\item By lemma \ref{id-EGS}, equation (\ref{equ:6}), the Ricci curvature satisfies $$\Delta_f\Ric(g)=-\Ric(g)-2\Rm(g)\ast\Ric(g).$$
Moreover, as $\lim_{+\infty}r_p^2\Ric(g)=0$, it implies $\lim_{+\infty}f\Ric(g)=0$ since $f$ is equivalent to $r_p^2/4$. In order to apply lemma \ref{coro-cruc-tens}, we have to prove that the curvature tensor decays appropriately at infinity. A priori, we have no information on the curvature growth at infinity, nonetheless, one can adapt the work of Munteanu and Wang \cite{Mun-Wan} to the expanding case by proving that the curvature tensor grows at most polynomially providing the Ricci curvature is bounded and $f$ is an exhausting function which is the case here. As $\Ric(g)$ decays quadratically, by lemma \ref{loc-est-tensor-egs}, $\nabla^2\Ric(g)$ grows at most like $f^{\alpha-1}$ if $\Rm(g)=\textit{O}(f^{\alpha})$ for some positive $\alpha$. Then, by using lemma \ref{id-op} as we did in the proof of theorem \ref{asy-con-gen}, one can show that $\Rm(g)=\textit{O}(f^{\alpha-1})$. Hence, by induction, one shows actually that $\Rm(g)=\textit{O}(f^{-1+\alpha'})$ for any $\alpha'\in(0,1)$.

  Now, we apply lemma \ref{coro-cruc-tens} with $T=\Ric(g)$, $\lambda=1$ and $\alpha'\in(0,1)$ to get 
\begin{eqnarray*}
\sup_{M^n}v^{\frac{n-k}{2}-1}e^v\arrowvert \nabla^k\Ric(g)\arrowvert<+\infty,\quad\forall k\geq 0.
\end{eqnarray*}

\item Clearly, by the previous step, the invariants $(\A_g^k(\Ric(g)))_{k\geq 0}$ are finite. Therefore, we can apply theorem \ref{asy-con-gen} to claim that $(M^n,g,\nabla f))$ is asymptotically conical to a metric cone $(C(X),dr^2+r^2g_X)$ with $(X,g_X)$ a smooth Riemannian manifold. Moreover, as $\lim_{+\infty}f\Ric(g)=0$, the metric cone is Ricci flat.

The convergence is exponential since, according to equation (\ref{est-rate-con-exp}), the rate is given by 
\begin{eqnarray*}
\textit{O}\left(\int_t^{+\infty}s^2\sup_{M_{s^2/4}}\arrowvert\Ric(g)\arrowvert\cdot s^{-3}ds\right)=\textit{O}\left(t^{-n}e^{-\frac{t^2}{4}}\right).
\end{eqnarray*}

\item
Concerning the lower bound for the scalar curvature, we employ the same method, that is, we compute the evolution of the quantity $v^{n/2-1}e^v\R_g$ hoping that some weighted laplacian of it is positive outside a compact set.

We first consider $e^v\R_g$ : 
\begin{eqnarray}
\Delta_v(e^v\R_g)&=&\left(\Delta_ve^v\right)\R_g+2<\nabla v,\nabla \left(\R_ge^v\right)>-2\arrowvert\nabla v\arrowvert^2\R_ge^v+e^v\Delta_v\R_g,\\
\Delta_{-v}(e^v\R_g)&=&(v-\arrowvert\nabla v\arrowvert^2-1)\R_ge^v-2\arrowvert\Ric(g)\arrowvert^2e^v\\ \label{evol-exp-scal}
&=&\left(\R_g+\frac{n}{2}-1\right)e^v\R_g-2\arrowvert\Ric(g)\arrowvert^2e^v.
\end{eqnarray}
Then, for some real $\alpha$,
\begin{eqnarray*}
\Delta_{-v}(v^{\alpha}e^v\R_g)&=&\Delta_{-v}\left(v^{\alpha}\right)(e^v\R_g)+2\alpha<\nabla \ln v,\nabla(v^{\alpha}e^v\R_g)>\\
&&-2\alpha^2\arrowvert\nabla\ln v\arrowvert^2v^{\alpha}e^v\R_g+\left(\R_g+\frac{n}{2}-1\right)v^{\alpha}e^v\R_g-2\arrowvert\Ric(g)\arrowvert^2v^{\alpha}e^v,\\
\Delta_{-v-2\alpha\ln v}(v^{\alpha}e^v\R_g)&=&\left(\frac{n}{2}-1-\alpha+\R_g+2\alpha\frac{\Delta v}{v}-\alpha(\alpha+1)\arrowvert\nabla\ln v\arrowvert^2\right)v^{\alpha}e^v\R_g\\
&&-2\arrowvert\Ric(g)\arrowvert^2v^{\alpha}e^v\\
&=&\left(\frac{n}{2}-1-\alpha+\frac{2\alpha\Delta v-\alpha(\alpha+1)}{v}+\alpha(\alpha+1)\frac{\Delta v}{v^2}\right)v^{\alpha}e^v\R_g\\
&&+(\R_g^2-2\arrowvert\Ric(g)\arrowvert^2)v^{\alpha}e^v.
\end{eqnarray*}
If $\alpha:=n/2-1$, we get
\begin{eqnarray*}
\Delta_{-v-(n-2)\ln v}(v^{n/2-1}e^v\R_g)&=&\left(\frac{(n/2-1)n/2+(n-2)\R_g}{v}\right)e^v\R_g\\
&&+\left(n/2(n/2-1)\frac{\R_g+n/2}{v^2}\right)v^{n/2-1}e^v\R_g\\
&&+(\R_g^2-2\arrowvert\Ric(g)\arrowvert^2)v^{n/2-1}e^v.
\end{eqnarray*}
Therefore, one has
\begin{eqnarray}
\Delta_{-v-(n-2)\ln v}(v^{n/2-1}e^v\R_g)\leq \left(\frac{C(n,\sup_{M^n}\R_g)}{v}+\R_g\right)v^{n/2-1}e^v\R_g.\label{evol-exp-pol-scal}
\end{eqnarray}

According to the differential inequality (\ref{evol-exp-pol-scal}), one has
$$\Delta_{-v-(n-2)\ln v}(v^{n/2-1}e^v\R_g)\leq \left(\frac{C(n,\sup_{M^n}\R_g,\sup_{M^n}v\R_g)}{v}\right)v^{n/2-1}e^v\R_g.$$
By multiplying by a function of the form $e^{Cv^{-1}}$ where $C$ is a constant depending on $n$, $\sup_{M^n}\R_g$ and $\sup_{M^n}v\R_g$ as we did before for the Ricci curvature, one has $$ \Delta_{-v-(n-2)\ln v-2Cv^{-1}}(v^{n/2-1}e^ve^{-Cv^{-1}}\R_g)<0,$$ outside a sublevel set of the form $\{v\leq C(n,\sup_{M^n}\R_g)\}$. Hence the result.
\end{enumerate}
\end{proof}

We decide to prove proposition \ref{ric-pin} as a natural follow-up to the proof of the previous theorem.

\begin{proof}[Proof of proposition \ref{ric-pin}]
\begin{enumerate}
\item As the Ricci curvature is pinched and the scalar curvature is nonnegative, the Ricci curvature is trivially nonnegative, therefore, the potential function is quadratic in the distance function by (\ref{inequ:3}) in proposition \ref{pot-fct-est}. Moreover, the Ricci curvature decays exponentially at infinity : this has already been shown in previous works on Ricci pinched expanders. We reprove it for the convenience of the reader. Consider the Morse flow $(\phi_t)_t$ associated to the potential function. Then, by (\ref{equ:2}) of lemma \ref{id-EGS},
\begin{eqnarray*}
\partial_t\R_g=\left<\nabla \R_g,\frac{\nabla f}{\arrowvert\nabla f\arrowvert^2}\right>=-2\Ric(g)(\textbf{n},\textbf{n}).
\end{eqnarray*}
 Since the Ricci curvature is pinched, 
 \begin{eqnarray*}
\partial_t\R_g\leq -\frac{2\epsilon}{n}\R_g,
\end{eqnarray*}
which implies that $\R_g(\phi_t(x))\leq e^{-2\epsilon t/n}\R_g(x)$, for $t\geq 0$, i.e. the Ricci curvature decays as $e^{-2\epsilon f/n}$. In particular, one has $\lim_{+\infty}r_p^2\Ric(g)=0$ and theorem \ref{dec-Ricci-flat-cone} applies. Moreover, as it is explained in lemma \ref{id-EGS}, the levels sets of $f$ are diffeomorphic to $\mathbb{S}^{n-1}$ : the link of the asymptotic cone of $(M^n,g,\nabla f)$ is therefore a smooth standard sphere endowed with an Einstein metric with constant $n-2$.\\

\item According to the previous step of the proof of proposition \ref{ric-pin}, the asymptotic cone is $(C(X),dr^2+r^2g_{X})$ where $X$ is  diffeomorphic to $\mathbb{S}^{n-1}$ and $g_X$ satisfies $\Ric(g_X)=(n-2)g_X$, which, in (global) dimension $3$ or $4$, implies that $(X,g_X)$ is isometric to the Euclidean sphere of curvature $1$. In particular, the asymptotic volume ratio is $$\AVR(g)=\lim_{r\rightarrow+\infty}\frac{\vol B(p,r)}{r^n}=\omega_n,$$
where $\omega_n$ is the volume of the unit ball of the $n$-dimensional Euclidean space. By the Bishop-Gromov theorem, $(M^n,g)$ is isometric to the Euclidean space when $n\in\{3,4\}$.\\

\item On the one hand, if the Ricci curvature is pinched, that is, there exists some positive number $\epsilon$ less than or equal to $1$ such that $$\Ric(g)\geq\frac{\epsilon \R_g}{n}g,$$
then we estimate the norm of the Ricci curvature as follows :
\begin{eqnarray*}
\arrowvert\Ric(g)\arrowvert^2&=&\left\arrowvert\Ric(g)-\frac{\epsilon \R_g}{n}g+\frac{\epsilon \R_g}{n}g\right\arrowvert^2\\
&\leq&(1-\epsilon)^2\R_g^2+\frac{2\epsilon (1-\epsilon)}{n}\R_g^2+\frac{\epsilon^2}{n}\R_g^2\\
&\leq&\left(c(n)\epsilon^2-2c(n)\epsilon+1\right)\R_g^2,
\end{eqnarray*}
where $c(n)=1-1/n$. Consequently, $2\arrowvert\Ric(g)\arrowvert^2\leq \R_g^2$ pointwisely if and only if $\epsilon\in[\epsilon_{-},1]$ where $$\epsilon_{-}:=1-\sqrt{1-\frac{1}{2\left(1-\frac{1}{n}\right)}},$$
which means that $e^v\R_g$ satisfies on $M^n$ according to equation (\ref{evol-exp-scal}),
\begin{eqnarray*}
\Delta_{-v}(e^v\R_g)\geq\left(\frac{n}{2}-1\right)e^v\R_g.
\end{eqnarray*}
By applying the maximum principle, we get that $(M^n,g)$ is scalar flat since $$\lim_{+\infty}e^v\R_g=0,$$ as soon as $n\geq 3$ by the Ricci curvature decay obtained in theorem \ref{dec-Ricci-flat-cone}. By the evolution equation (\ref{equ:7}), $(M^n,g)$ is Ricci flat hence isometric to the Euclidean case according to \cite{Pet-Wyl-Rig}.
\end{enumerate}
\end{proof}

\subsection{Estimates in dimension $3$}

Define $$T:=\Ric(g) -\frac{R}{2}(g-\mathbf{n}\otimes\mathbf{n}),$$
where $\textbf{n}:=\nabla f/\arrowvert\nabla f\arrowvert$.
This tensor reflects the conical geometry of a three dimensional manifold. A priori, if the curvature decays quadratically, $T$ decays quadratically too. The next lemma which is independent of the dimension tells us that $T$ should converge much faster.
\begin{lemma}\label{dec-rap-ric}
Let $(M^n,g,\nabla f)$ be an expanding gradient Ricci soliton such that $\A^0_g(\Ric(g))$ and $\A_g^1(\R_g)$ are finite. Then,
$\sup_{M_t}\arrowvert\Ric(g)(\mathbf{n})\arrowvert=\textit{O}(f^{-2})$, where $\mathbf{n}:=\frac{\nabla f}{\arrowvert\nabla f\arrowvert}$.
\end{lemma} 

\begin{proof}
By \eqref{equ:2}, one has 
\begin{eqnarray*}
\Ric(g)(\mathbf{n})=-\frac{\nabla \R_g}{2\arrowvert \nabla f\arrowvert}.
\end{eqnarray*}
We conclude by using the assumption $\A_g^1(\R_g)<+\infty$.

\end{proof}
\begin{rk}
Again, this is consistent with the conical geometry of such an expanding gradient Ricci soliton. Indeed, in the case of a metric cone, radial curvatures vanish. In particular, they decay faster than spherical curvatures which decay quadratically.
\end{rk}

The following proposition tells us that $T$ actually decays much faster in any direction.

\begin{prop}\label{est-sge-3d}
Let $(M^3,g,\nabla f)$ be an expanding gradient Ricci soliton. Then,
\begin{eqnarray*}
\A^0_g(\Ric(g))<+\infty+\A^1_g(\R_g)<+\infty \Rightarrow T=\textit{O}(f^{-2}).
\end{eqnarray*}
Moreover,
\begin{eqnarray*}
\A^0_g(\Ric(g))<+\infty+\A_g^k(\R_g)<+\infty,\forall k\geq 0\Rightarrow  \nabla^k T=\textit{O}(f^{-2-k/2}), \forall k\geq 0.
\end{eqnarray*}

In particular,
\begin{eqnarray*}
\A^0_g(\Ric(g))<+\infty+\A_g^k(\R_g)<+\infty,\forall k\geq 0\Rightarrow \A_g^k(\Ric(g))<+\infty,\forall k\geq 0.
\end{eqnarray*}
\end{prop}

\begin{proof}
Since the dimension is three :
\begin{eqnarray*}
\Rm_{\mathbf{n}ij\mathbf{n}}=\Ric(g)_{\mathbf{n}\mathbf{n}}g_{ij}+\Ric(g)_{ij}-\Ric(g)_{\mathbf{n}j}g_{\mathbf{n}i}-\Ric(g)_{\mathbf{n}i}g_{\mathbf{n}j}-\frac{\R_g}{2}\left(g_{ij}-g_{\mathbf{n}i}g_{\mathbf{n}j}\right).
\end{eqnarray*}
Now, by lemma \ref{id-EGS} and by the very definition of $T$ :
\begin{eqnarray*}
-\frac{1}{\arrowvert \nabla f\arrowvert}\div\Rm_{\mathbf{n}ij}=\Ric(g)_{\mathbf{n}\mathbf{n}}g_{ij}-\Ric(g)_{\mathbf{n}j}g_{\mathbf{n}i}-\Ric(g)_{\mathbf{n}i}g_{\mathbf{n}j}+T_{ij},
\end{eqnarray*}
At this stage in the proof, $T$ decays as expected assuming $\A^0_g(\Ric(g))$ and $\A^1_g(\Ric(g))$ are finite by lemma \ref{dec-rap-ric}. We claim that it is sufficient to assume the finiteness of $\A^1_g(\R_g)$ instead of $\A^1_g(\Ric(g))$.

Indeed,
\begin{eqnarray*}
\frac{1}{\arrowvert \nabla f\arrowvert}[\nabla_{\mathbf{n}}\Ric(g)_{ij}-\nabla_i\Ric(g)_{\mathbf{n}j}]+T_{ij}=-\Ric(g)_{\mathbf{n}\mathbf{n}}g_{ij}+\Ric(g)_{\mathbf{n}j}g_{\mathbf{n}i}+\Ric(g)_{\mathbf{n}i}g_{\mathbf{n}j}.
\end{eqnarray*}
Hence,
\begin{eqnarray*}
2\nabla_{\frac{\mathbf{n}}{\arrowvert \nabla f\arrowvert}}\Ric(g)_{ij}+2T_{ij}&=&2[-\Ric(g)_{\mathbf{n}\mathbf{n}}g_{ij}+\Ric(g)_{\mathbf{n}j}g_{\mathbf{n}i}+\Ric(g)_{\mathbf{n}i}g_{\mathbf{n}j}]\\
&&+\frac{1}{\arrowvert \nabla f\arrowvert}[\nabla_i\Ric(g)_{\mathbf{n}j}+\nabla_j\Ric(g)_{\mathbf{n}i}].
\end{eqnarray*}
Therefore,
\begin{eqnarray*}
2\nabla_{\frac{\mathbf{n}}{\arrowvert \nabla f\arrowvert}}\left[\Ric(g)_{ij}-\frac{\R_g}{2}g_{ij}\right]+2T_{ij}&=&2[\Ric(g)_{\mathbf{n}j}g_{\mathbf{n}i}+\Ric(g)_{\mathbf{n}i}g_{\mathbf{n}j}]\\
&&+\frac{1}{\arrowvert \nabla f\arrowvert}[\nabla_i\Ric(g)_{\mathbf{n}j}+\nabla_j\Ric(g)_{\mathbf{n}i}].
\end{eqnarray*}
Now,
\begin{eqnarray*}
\nabla_i\Ric(g)_{\mathbf{n}j}+\nabla_j\Ric(g)_{\mathbf{n}i}&=&-\frac{\nabla^2\Ric(g)_{ij}}{\arrowvert \nabla f\arrowvert}+\frac{1}{2}\nabla_j\R_g\cdot\frac{\nabla_i\arrowvert \nabla f\arrowvert}{\arrowvert \nabla f\arrowvert^2}+\frac{1}{2}\nabla_i\R_g\cdot\frac{\nabla_j\arrowvert \nabla f\arrowvert}{\arrowvert \nabla f\arrowvert^2}\\
&&-\Ric(g)(\nabla_i\mathbf{n},j)-\Ric(g)(\nabla_j\mathbf{n},i).
\end{eqnarray*}
So that,
\begin{eqnarray*}
2[\nabla_{\frac{\mathbf{n}}{\arrowvert \nabla f\arrowvert}}T+T]&=&\frac{\nabla_{\nabla f}\R_g}{\arrowvert \nabla f\arrowvert^2}\mathbf{n}\otimes\mathbf{n}+\frac{\R_g}{\arrowvert \nabla f\arrowvert^2}\Sym[\Ric(g)^{\perp}(\mathbf{n})\otimes\mathbf{n}]+2\Sym[\Ric(g)(\mathbf{n})\otimes\mathbf{n}]\\
&&+\frac{1}{\arrowvert \nabla f\arrowvert^2}\left[-\nabla^2\R_g+\frac{1}{2}\Sym[\nabla \R_g\otimes\nabla\ln\arrowvert \nabla f\arrowvert]-\Sym[\Ric(g)\otimes\arrowvert \nabla f\arrowvert\nabla\mathbf{n}]\right]\\
&=&\frac{\nabla_{\mathbf{n}}\R_g}{\arrowvert \nabla f\arrowvert}\mathbf{n}\otimes\mathbf{n}-\frac{1}{2}\frac{\R_g}{\arrowvert \nabla f\arrowvert^3}\Sym[\nabla^{\perp}\R_g\otimes\mathbf{n}]-\frac{1}{\arrowvert \nabla f\arrowvert}\Sym[\nabla \R_g\otimes\mathbf{n}]\\
&&+\frac{1}{\arrowvert \nabla f\arrowvert^2}\left[-\nabla^2\R_g+\frac{1}{2}\Sym[\nabla \R_g\otimes\nabla\ln\arrowvert \nabla f\arrowvert]-\Sym[\Ric(g)\otimes\arrowvert \nabla f\arrowvert\nabla\mathbf{n}]\right]
\end{eqnarray*}
A priori, by lemma \ref{elliptic}, as $\A_g^0(\Ric(g))$ is finite, $\nabla^2\R_g$ decays quadratically in the distance to a fixed point. Therefore, 
\begin{eqnarray}
\partial_tT+T=\textit{O}(f^{-2}).\label{evo-T-dim-3}
\end{eqnarray}
We conclude the first part of proposition \ref{est-sge-3d} by integrating this differential equality along the Morse flow generated by $\nabla f/\nfa^2$.

The second part of proposition \ref{est-sge-3d} follows by derivating (\ref{evo-T-dim-3}) to get, for any nonnegative integer $k$,
\begin{eqnarray*}
\partial_t\nabla^kT+\nabla^kT=\textit{O}(f^{-2-k/2}),
\end{eqnarray*}
by using the finiteness of all the invariants $(\A_g^i(\R_g))_{i\geq 0}$. Again, by integrating this differential equality, one gets $\nabla^kT=\textit{O}(f^{-2-k/2}).$\\

Finally, by the very definition of $T$, $$\Ric(g)=T+\frac{\R_g}{2}(\textbf{n}\otimes\textbf{n}),$$
which implies by the previous step that, for $k\geq 1$,
\begin{eqnarray*}
\limsup_{+\infty}f^{\frac{k+2}{2}}\arrowvert\nabla^k\Ric(g)\arrowvert\leq c_k\limsup_{+\infty}f^{\frac{k+2}{2}}\arrowvert\nabla^k\R_g\arrowvert<\infty.
\end{eqnarray*}

\end{proof}

\subsection{Estimates in higher dimensions : elliptic constraints at infinity} 

In this section, we investigate the asymptotic geometry of Ricci expanders whose asymptotic cones have an Einstein link. We first consider the case where the convergence is only $C^3$ :

\begin{theo}\label{dec-ric-c^3}
Let $(M^n,g,\nabla f)$ be an expanding gradient Ricci soliton. Assume the convergence to its asymptotic cone is at least $C^3$.
If the metric of the link has constant Ricci curvature, then
\begin{eqnarray*}
\Ric(g)-\frac{\R_g}{(n-1)}(g-\mathbf{n}\otimes\mathbf{n})=\textit{O}(f^{-3/2})
\end{eqnarray*}
\end{theo}

\begin{proof}
Define $$T:=\Ric(g)-\frac{R_g}{(n-1)}(g-\mathbf{n}\otimes\mathbf{n}).$$ $T$ is well-defined outside a compact set (or even a point if $\Ric(g)\geq-(\delta/2)g$ for $\delta\in[0,1)$). Then, $\tr T=0$ and 
\begin{eqnarray*}
\Delta_f T+T&=&-2\Rm(g)\ast\Ric(g)+\frac{2}{n-1}\arrowvert\Ric(g)\arrowvert^2(g-\mathbf{n}\otimes\mathbf{n})\\
&&+\frac{\R_g}{n-1}\Delta_f(\mathbf{n}\otimes\mathbf{n})+\frac{2}{n-1}\nabla\R_g\cdot\nabla(\mathbf{n}\otimes\mathbf{n}).
\end{eqnarray*}
Now, recall that $\Rm(g)\ast\Ric(g)_{ij}:=\Rm(g)_{iklj}\Ric(g)_{kl}.$ Therefore,
\begin{eqnarray*}
<\Rm(g)\ast\Ric(g),T>&=&<\Rm(g)\ast T,T>+\frac{\R_g}{n-1}<\Rm(g)\ast(g-\mathbf{n}\otimes\mathbf{n}),T>\\
&=&<\Rm(g)\ast T,T>+\frac{\R_g}{n-1}<\Ric(g)-\Rm(g)(\mathbf{n},\cdot,\cdot,\mathbf{n}),T>\\
&=&<\Rm(g)\ast T,T>+\frac{\R_g}{n-1}\arrowvert T\arrowvert^2-\frac{\R_g^2}{(n-1)^2}T(\mathbf{n},\mathbf{n})\\
&&-\frac{\R_g}{n-1}<\Rm(g)(\mathbf{n},\cdot,\cdot,\mathbf{n}),T>.
\end{eqnarray*}
Now, according to lemma \ref{id-EGS}, we know that 
\begin{eqnarray*}
<\Rm(g)(\mathbf{n},\cdot,\cdot,\mathbf{n}),T>&=&\textit{O}(v^{-2})\arrowvert T\arrowvert\\
T(\mathbf{n},\mathbf{n})&=&\Ric(\mathbf{n},\mathbf{n})=\textit{O}(v^{-2})\\
\arrowvert T\arrowvert^2&=&\arrowvert\Ric(g)\arrowvert^2-\frac{\R_g^2}{n-1}+\frac{2\R_g}{n-1}\Ric(\mathbf{n},\mathbf{n}),
\end{eqnarray*}
Hence,
\begin{eqnarray*}
\frac{1}{2}\left(\Delta_f\arrowvert T\arrowvert^2+2\arrowvert T\arrowvert^2-2\arrowvert\nabla T\arrowvert^2\right)&=&-2<\Rm(g)\ast T,T>-\frac{2}{n-1}(\R_g+\Ric(g)(\mathbf{n},\mathbf{n}))\arrowvert T\arrowvert^2\\
&&+\frac{4\R_g}{(n-1)^2}\Ric(\mathbf{n},\mathbf{n})^2+\frac{2\R_g}{n-1}\Rm(g)_{\textbf{n}..\textbf{n}}\ast T\\
&&+\frac{\R_g}{n-1}<\Delta_f(\mathbf{n}\otimes\mathbf{n}),T>+\frac{2}{n-1}<\nabla\R_g\cdot\nabla(\mathbf{n}\otimes\mathbf{n}),T>.
\end{eqnarray*}
Now, since, $$\nabla\textbf{n}=\frac{\nabla^2f-\nabla^2f(\textbf{n})\otimes\textbf{n}}{\arrowvert\nabla f\arrowvert},$$
\begin{eqnarray*}
<\nabla\R_g\cdot\nabla(\mathbf{n}\otimes\mathbf{n}),T>&=&2<\nabla_{\nabla \R_g}\mathbf{n},\Ric(g)(\mathbf{n})>\\
&=&\textit{O}(v^{-4}),\\
<\nabla_{\nabla f}(\mathbf{n}\otimes\mathbf{n}),T>&=&2<\nabla_{\nabla f}\mathbf{n},\Ric(g)(\mathbf{n})>\\
&=&2<\Ric(g)(\textbf{n})-\Ric(g)(\textbf{n},\textbf{n})\textbf{n},\Ric(g)(\mathbf{n})>=\textit{O}(v^{-4}),\\
<\Delta(\mathbf{n}\otimes\mathbf{n}),T>&=&2<\Delta(\mathbf{n}),\Ric(g)(\mathbf{n})>+2<\nabla\mathbf{n}\ast\nabla\mathbf{n},T>\\
&=&\textit{O}(v^{-1})\arrowvert T\arrowvert.
\end{eqnarray*}
Finally, we get : 
\begin{eqnarray*}
\frac{1}{2}\left(\Delta_v\arrowvert T\arrowvert^2+2\arrowvert T\arrowvert^2-2\arrowvert\nabla T\arrowvert^2\right)&=&\textit{O}(v^{-1})\arrowvert T\arrowvert^2+\textit{O}(v^{-2})\arrowvert T\arrowvert+\textit{O}(v^{-4})\\
&=&\textit{O}(v^{-1})\arrowvert T\arrowvert^2+\textit{O}(v^{-3}),
\end{eqnarray*}
where we used the Young inequality in the last line. Define now $u:=\arrowvert T\arrowvert^2$. Then, there are positive constants $c_0, c_1$ such that
\begin{eqnarray*}
\Delta_vu+2u\geq -c_0v^{-1}u-c_1v^{-3}.
\end{eqnarray*}
 We apply lemma \ref{lemm-cruc} to $u$, $\lambda=-2$, $\alpha=1$ and $Q=c_1v^{-3}$ to get that $\sup_{M^n}v^{3}u<+\infty$ since $\lim_{+\infty}v^2u=0$ by assumption on $T$, which is excatly the expected decay on $T$.

\end{proof}

The decay of the tensor $T$ considered in theorem \ref{dec-ric-c^3} can be sharpened as soon as the convergence to the asymptotic cone is at least $C^4$ :
\begin{theo}\label{dec-ric-c^4}
Let $(M^n,g,\nabla f)$ be an expanding gradient Ricci soliton. Assume the convergence to its asymptotic cone is at least $C^4$.
\begin{enumerate}
\item If the metric of the link has constant scalar curvature $(n-1)(n-2)+\R_{\infty}$, then
\begin{eqnarray*}
\R_g=\frac{\R_{\infty}}{f}+\textit{O}(f^{-2}),
\end{eqnarray*}

\item
If the metric of the link has constant Ricci curvature, then
\begin{eqnarray*}
T:=\Ric(g)-\frac{\R_g}{(n-1)}(g-\mathbf{n}\otimes\mathbf{n})=\textit{O}(f^{-2})
\end{eqnarray*}
\item If the metric of the link has constant curvature, then
\begin{eqnarray*}
\Rm(g)-\frac{R_g}{(n-1)(n-2)}(g-\mathbf{n}\otimes\mathbf{n})\odot(g-\mathbf{n}\otimes\mathbf{n})=\textit{O}(f^{-2}),
\end{eqnarray*}
\end{enumerate}
\end{theo}

\begin{proof}

First of all, denote by $\R_{\infty}$ the limit of the rescaled curvature, i.e. $\R_{\infty}:=\lim_{+\infty}v\R_g$. Then, by the soliton identities given by lemma \ref{id-EGS},
\begin{eqnarray*}
\nabla_{\nabla f}(v\R_g-\R_{\infty})&=&\arrowvert\nabla v\arrowvert^2\R_g+v\nabla_{\nabla f}\R_g\\
&=&-\Delta v\R_g-v(\Delta\R_g+2\arrowvert\Ric(g)\arrowvert^2)\\
&=&\textit{O}(v^{-1}).
\end{eqnarray*}
In particular, if we integrate the previous estimate along the flow generated by $\nabla f/\arrowvert\nabla f\arrowvert^2$, then we get, 
\begin{eqnarray*}
v\R_g-\R_{\infty}=\textit{O}(v^{-1}).
\end{eqnarray*}

The same procedure applies to the Ricci curvature if the metric of the link of the asymptotic cone has constant Ricci curvature. Indeed, by the previous estimate on the scalar curvature,
\begin{eqnarray*}
vT=v\Ric(g)-\frac{\R_{\infty}}{n-1}(g-\mathbf{n}\otimes\mathbf{n})+\textit{O}(v^{-1}).
\end{eqnarray*}
Therefore, it suffices to prove that $$v\Ric(g)-\frac{\R_{\infty}}{n-1}(g-\mathbf{n}\otimes\mathbf{n})=\textit{O}(v^{-1}).$$
Again, by lemma \ref{id-EGS},
\begin{eqnarray*}
\nabla_{\nabla f}\left(v\Ric(g)-\frac{\R_{\infty}}{n-1}(g-\mathbf{n}\otimes\mathbf{n})\right)&=&-v(\Delta\Ric(g)+2\Rm(g)\ast\Ric(g))\\&&-\Delta v\Ric(g)+\frac{\R_{\infty}}{n-1}\begin{bmatrix}(\nabla_{\nabla f}\mathbf{n})\otimes\mathbf{n}+\mathbf{n}\otimes(\nabla_{\nabla f}\mathbf{n})\end{bmatrix}\\
&=&\textit{O}(v^{-1}).
\end{eqnarray*}
The result follows by integrating along the Morse flow of the potential function $f$. The same procedure applies to the curvature tensor if the metric of the link of the asymptotic cone has constant curvature.
\end{proof}

\subsection{Positively curved expanding gradient Ricci solitons}
Even if proposition \ref{inf-bound-scal} belongs to global estimates, we find it more convenient to give its proof now :

\begin{proof}[Proof of proposition \ref{inf-bound-scal}]
Recall first that $\Delta_fv=v$. Then, by equation (\ref{equ:7}), one has
\begin{eqnarray*}
\Delta_f(v\R_g)=-2\arrowvert\Ric(g)\arrowvert^2v+2<\nabla(v\R_g),\nabla\ln v>-2(v\R_g)\arrowvert\nabla\ln v\arrowvert^2\leq 0.
\end{eqnarray*}
One gets the expected estimate by applying the maximum principle to sub level set $\{v\leq t\}$ with $t$ tending to $+\infty$.

\end{proof}

\begin{rk}
\begin{itemize}
\item The proof of proposition \ref{inf-bound-scal} does not require the Ricci curvature to be nonnegative, one only needs nonnegative scalar curvature plus the fact that $v$ is an exhaustion function.
\item In particular, proposition \ref{inf-bound-scal} gives a positive lower bound for the rescaled scalar curvature as soon as the metric $g_X$ of the link of the asymptotic cone $C(X)$ satisfies $\R_{g_X}>(n-1)(n-2)$ on $X$.
\end{itemize}
\end{rk}

The next proposition gives an asymptotic lower bound on the Ricci curvature in case the asymptotic cone is sufficiently positively curved : 
\begin{prop}\label{inf-bound-ric}
Let $(M^n,g,\nabla f)$ be an expanding gradient Ricci soliton with positive Ricci curvature that converges (in the $C^4$ sense at least) to its asymptotic cone $(C(X),dr^2+r^2g_X)$ where $\Ric(g_X)> (n-2)g_X$. Then, for some $p\in M^n$,

\begin{eqnarray*}
\liminf_{r\rightarrow +\infty}r^4\inf_{\partial B(p,r)}\Ric(g)>0.
\end{eqnarray*}
\end{prop}

\begin{proof}

By assumption on the positivity of the Ricci curvature of the metric of the link of the cone, it is sufficient to prove it for the radial direction. By the soliton identity (\ref{equ:2}),

 \begin{eqnarray*}
2v^2 \Ric(g)(\mathbf{n},\mathbf{n})&=&-\frac{v^2}{\arrowvert \nabla f\arrowvert^2}<\nabla \R_g,\nabla f>\\
&=&\frac{v^2}{\arrowvert\nabla f\arrowvert^2}\left(\Delta \R_g+2\arrowvert\Ric (g)\arrowvert^2+\R_g\right).
\end{eqnarray*}
In particular, by using the decay of $\Delta\R_g$ and the fact that $\liminf_{+\infty}v\R_g>0$, one gets $\inf_{M^n}v^2\Ric(g)(\mathbf{n},\mathbf{n})>0$.

Hence the result.

\end{proof}

We end this subsection by proving the analogue of proposition \ref{inf-bound-ric} for the curvature operator.

\begin{prop}\label{inf-bound-op}
Let $(M^n,g,\nabla f)$ be an expanding gradient Ricci soliton with positive curvature operator that converges (in the $C^4$ sense) to its asymptotic cone $(C(X),dr^2+r^2g_X)$ where $\Rm(g_X)> \Id_{\Lambda^2TX}$. Then, for $p\in M^n$,

\begin{eqnarray*}
\liminf_{r\rightarrow +\infty}r^4\inf_{\partial B(p,r)}\Rm(g)>0.
\end{eqnarray*}

\end{prop}
\begin{proof}
 Any bivector $\alpha\in \Lambda^2TM$ can be written as $\alpha=\alpha'+\beta\wedge\mathbf{n}$, where $\alpha'$ is a linear combination of bivectors spanning $2$-planes orthogonal to $\mathbf{n}$ and $\beta$ is a vector orthogonal to $\mathbf{n}$. Then, by the soliton identity (\ref{equ:4})
 \begin{eqnarray*}
\Rm(g)(\alpha,\alpha)&=&\Rm(g)(\alpha',\alpha')+2\Rm(g)(\alpha',\beta\wedge\mathbf{n})+\Rm(g)(\beta\wedge\mathbf{n},\beta\wedge\mathbf{n})\\
&\geq&\Rm(g)(\alpha',\alpha')-c(n)v^{-1/2}\arrowvert\div\Rm(g)\arrowvert\arrowvert\alpha'\arrowvert\arrowvert\beta\arrowvert+\Rm(g)(\beta,\mathbf{n},\mathbf{n},\beta).
\end{eqnarray*}
 Now,

\begin{eqnarray*}
\Rm(g)(\nabla f,U,V,\nabla f)=-\div\Rm(g)(\nabla f, U, V)=-\nabla_{\nabla f}\Ric(g)(U,V)+\nabla_U\Ric(g)(\nabla f,V),
\end{eqnarray*}
for any vector $U$ and $V$.
On the other hand, by the soliton identities given by  lemma \ref{id-EGS},
\begin{eqnarray*}
\nabla_U\Ric(g)(\nabla f,V)&=&U\cdot\Ric(g)(\nabla f,V)-\Ric(g)(\nabla_U\nabla f,V)-\Ric(g)(\nabla f,\nabla_UV)\\
&=&<\nabla_U(\Ric(g)(\nabla f)),V>-\Ric(g)\left(\frac{U}{2}+\Ric(g)(U),V\right)\\
&=&-\frac{\nabla^2\R_g}{2}(U,V)-\Ric(g)\left(\frac{U}{2}+\Ric(g)(U),V\right).
\end{eqnarray*}
Therefore,
\begin{eqnarray*}
\Rm(g)(\mathbf{n},U,V,\mathbf{n})&=&\frac{1}{2\arrowvert\nabla f\arrowvert^2}\left(2\Delta\Ric(g)+\Ric(g)+4\Rm(g)\ast\Ric(g)\right)(U,V)\\
&&-\frac{1}{2\arrowvert\nabla f\arrowvert^2}\left(\nabla^2\R_g(U,V)+2\Ric(g)\otimes\Ric(g)\right)(U,V).
\end{eqnarray*}
Hence, by the decay of the curvature and the assumption on the positivity of the curvature operator of the metric of the link of the cone,  $$\inf_{x\in M^n}v^2(x)\min_{\beta\perp\mathbf{n}}\Rm(g)(x)(\beta\wedge\mathbf{n},\beta\wedge\mathbf{n})>0.$$
Finally, we get, by the Young inequality,

\begin{eqnarray*}
\Rm(g)(\alpha,\alpha)&\geq&\frac{C}{v}\arrowvert\alpha'\arrowvert^2-\textit{O}(v^{-2})\arrowvert\alpha'\arrowvert\arrowvert\beta\arrowvert+\frac{C}{v^2}\arrowvert\beta\arrowvert^2\\
&\geq&\frac{C}{v}\left(1-\textit{O}(v^{-1})\right)\arrowvert\alpha'\arrowvert^2+\frac{C}{v^2}\arrowvert\beta\arrowvert^2,
\end{eqnarray*}
where $C$ is a positive constant space independent which can vary from line to line. Hence the expected lower bound for the curvature operator.

\end{proof}

\section{A priori curvature estimates and compactness theorems}\label{comp-theo-sec}

We start by establishing an a priori estimate for general subsolutions of weighted elliptic equations with quadratic nonlinearities on a Ricci expander :

\begin{prop}\label{a-priori-reaction}
Let $(M^n,g,\nabla f)$ be an expanding gradient Ricci soliton with bounded scalar curvature and such that the potential function $f$ is an exhausting function. Assume there exists a bounded nonnegative function $u:M\rightarrow \mathbb{R}$ such that
\begin{eqnarray*}
\Delta_fu\geq-u-u^2,\quad\limsup_{+\infty}vu<+\infty.
\end{eqnarray*}

Then,
\begin{eqnarray*}
\sup_{M^n}vu\leq C(\sup_{M^n}u,\limsup_{+\infty}vu).
\end{eqnarray*}

\end{prop}

\begin{rk}
The coefficient in front of the linear term matters whereas the one in front of the quadratic one does not. Actually, the coefficient in front of the linear term dictates the decay of this solution at infinity.
\end{rk}

\begin{proof}
First of all, as the potential function is  exhausting, the level sets $\{f=t\}$ of the potential function are smooth compact diffeomorphic hypersurfaces as soon as $t>\sup_{M^n}\R_g$ by lemma \ref{id-EGS}.

As $\Delta_fv=v$,
\begin{eqnarray*}
\Delta_f(vu)&\geq& -\frac{(vu)^2}{v}+2<\nabla(vu),\nabla \ln v>-2(vu)\arrowvert\nabla\ln v\arrowvert^2\\
&\geq&2<\nabla(vu),\nabla \ln v>-\frac{(vu+1)^2}{v}+\frac{1}{v},
\end{eqnarray*}
 by proposition \ref{pot-fct-est} Now, define $w:=(vu+1)^{-1}$. Then,
\begin{eqnarray*}
\Delta_{v-2 \ln v}w&=&-\frac{\Delta_{v-2 \ln v}(vu)}{(vu+1)^2}+2\frac{\arrowvert\nabla(vu)\arrowvert^2}{(vu+1)^3}\leq\frac{1}{v}+2\arrowvert\nabla w\arrowvert^2(uv+1)\\
&\leq&\frac{1}{v}+2\arrowvert\nabla w\arrowvert^2w^{-1}.
\end{eqnarray*}
Finally, consider $W:=w+2v^{-1}$. Then, as
\begin{eqnarray*}
\Delta_{v-2 \ln v}v^{-1}&=&-\frac{\Delta_{v-2\ln v}v}{v^2}+2\arrowvert\nabla v\arrowvert^2v^{-3}\\
&=&-\frac{1}{v}+4\arrowvert\nabla v\arrowvert^2v^{-3},
\end{eqnarray*}
we get,
\begin{eqnarray*}
\Delta_{v-2 \ln v}W&\leq& -v^{-1}+8\arrowvert\nabla v\arrowvert^2v^{-3}+2\arrowvert\nabla w\arrowvert^2w^{-1}\\
&\leq&-v^{-1}+8v^{-2}+2\arrowvert\nabla w\arrowvert^2w^{-1}.
\end{eqnarray*}
Assume that $W$ attains its minimum at an interior point in the slice $\{t_1\leq f\leq t_2\}$ with $t_1$ to be chosen later. Then, at this point, one has by the previous differential inequality :
\begin{eqnarray*}
0\leq -v^{-1}+8v^{-2}+2\arrowvert\nabla2v^{-1}\arrowvert^2w^{-1}.
\end{eqnarray*}
That is,
\begin{eqnarray*}
1-8v^{-1}\leq8(\sup_{M^n}u+v^{-1})v^{-1},
\end{eqnarray*}
which is impossible if $t_1>C(\sup_{M^n}u)$. Therefore, $W$ attains its minimum on the boundary of $\{t_1\leq f\leq t_2\}$ as soon as $t_1$ is larger than $C(\sup_{M^n}u)>0$. By letting $t_2$ go to $+\infty$, one has 
\begin{eqnarray*}
\min_MW\geq\min\left(C(\sup_{M^n}u),\liminf_{+\infty}W\right),
\end{eqnarray*}
which turns out to be exactly the expected estimate.

\end{proof}

It turns out that proposition \ref{a-priori-reaction} can be applied to the full curvature tensor and the scalar curvature when the Ricci curvature is nonnegative as the next two lemmata show.

\begin{coro}\label{a-prio-bd-curv-tensor}
Let $(M^n,g,\nabla f)$ be an expanding gradient Ricci soliton with finite asymptotic curvature ratio, i.e. $\A_g^0(\Rm(g))<+\infty$. Then,
\begin{eqnarray*}
\sup_{M^n}v\arrowvert\Rm(g)\arrowvert\leq C\left(n,\sup_{M^n}\arrowvert\Rm(g)\arrowvert,\A_g^0(\Rm(g))\right).
\end{eqnarray*}
\end{coro}

\begin{proof}
First of all, as $\A_g^0(\Rm(g))<+\infty$, it can be shown that $f$ is equivalent to $r_p^2/4$, in particular $f$ is proper.

By equation (\ref{a-priori-reaction}), the curvature operator $\Rm(g)$ of an expanding gradient Ricci soliton satisfies
\begin{eqnarray*}
\Delta_f\Rm(g)=-\Rm(g)+\Rm(g)\ast\Rm(g).
\end{eqnarray*}
Therefore,
\begin{eqnarray*}
\Delta_f\arrowvert\Rm(g)\arrowvert^2\geq2\arrowvert\nabla\Rm(g)\arrowvert^2-2\arrowvert\Rm(g)\arrowvert^2-c(n)\arrowvert\Rm(g)\arrowvert^3.
\end{eqnarray*}
Now, since the curvature may vanish, we have to deal with $u_{\epsilon}:=\sqrt{\arrowvert\Rm(g)\arrowvert^2+\epsilon^2}$ first, where $\epsilon>0$. By routine computations, one has
\begin{eqnarray*}
\Delta_fu_{\epsilon}\geq-u_{\epsilon}-c(n)u_{\epsilon}^2.
\end{eqnarray*}
Using the proof of proposition \ref{a-priori-reaction}, one gets, with the same notations, $W_{\epsilon}:=(vu_{\epsilon}+1)^{-1}$,
\begin{eqnarray}\label{min-W}
\min_{t_1\leq f\leq t_2}W_{\epsilon}=\min\left\{\min_{f=t_1} W_{\epsilon},\min_{f=t_2}W_{\epsilon}\right\},
\end{eqnarray}
for $t_1$ larger than a positive constant depending on $\sup_{M^n}u_{\epsilon}$ which does not depend on $\epsilon$ if $\epsilon$ is less than $1$. Therefore, if $\epsilon$ goes to $0$ in (\ref{min-W}), and by letting $t_2$ go to $+\infty$, we get the expected estimate.

\end{proof}

\begin{coro}\label{scal-max-ppe-infty}
Let $(M^n,g,\nabla f)$ be an expanding gradient Ricci soliton with nonnegative Ricci curvature such that $\limsup_{+\infty}v\R_g<+\infty$. Then,
\begin{eqnarray*}
\sup_{M^n}v\R_g\leq C\left(\sup_{M^n}\R_g,\A_g^0(\R_g)\right).
\end{eqnarray*}
\end{coro}

\begin{proof}
By proposition \ref{pot-fct-est}, $f$  is an exhausting function and the scalar curvature is bounded. Moreover, the scalar curvature satisfies the following differential inequality since the Ricci curvature is nonnegative :
\begin{eqnarray*}
\Delta_f\R_g=-\R_g-2\arrowvert\Ric(g)\arrowvert^2\geq -\R_g-2\R_g^2.
\end{eqnarray*}
Hence the result by applying proposition \ref{a-priori-reaction} to the scalar curvature $\R_g$.
\end{proof}

As we are dealing with smooth convergence, we need to control globally the rescaled covariant derivatives of the curvature tensor in terms of the asymptotic curvature ratios $(\A_g^k(\Rm(g)))_{k\geq 0}$ and a bound on the curvature tensor. We state and prove the following proposition dealing only with the full curvature tensor. 

\begin{prop}\label{curv-op-max-ppe-infty}
Let $(M^n,g,\nabla f)$ be an expanding gradient Ricci soliton such that $$\A_g^i(\Rm(g))<+\infty,$$ for $i\in\{0,...,k\}$. Then,
\begin{eqnarray*}
\sup_{M^n}\left(v^{1+k/2}\arrowvert\nabla^k\Rm(g)\arrowvert\right)\leq C\left(n,k,\min_{M^n}v,\sup_{M^n}\arrowvert\Rm(g)\arrowvert,\A_g^0(\Rm(g)),...,\A_g^k(\Rm(g))\right).
\end{eqnarray*}

\end{prop}

\begin{proof}[Proof of proposition \ref{curv-op-max-ppe-infty}]

We cannot apply directly corollary \ref{coro-cruc-tens}, since it assumes a priori bounds on the full rescaled covariant derivatives of the curvature.
 
Nonetheless, we prove these estimates in a similar way by induction on $k$.

We begin by deriving a  nice differential inequality satisfied by $\nabla^k\Rm(g)$ :
\begin{eqnarray*}
\Delta_f\arrowvert\nabla^k\Rm(g)\arrowvert^2&=&2\arrowvert\nabla^{k+1}\Rm(g)\arrowvert^2+2<[\Delta,\nabla^k]\Rm(g),\nabla^k\Rm(g)>\\
&&+2<[\nabla_{\nabla f},\nabla^k]\Rm(g),\nabla^k\Rm(g)>\\
&&+2<\nabla^k(-\Rm(g)+\Rm(g)\ast\Rm(g)),\nabla^k\Rm(g)>.
\end{eqnarray*}

According to (\ref{comm-1}) and (\ref{comm-2}) , we get for $k\geq 2$, (the case $k=1$ can be proved in a more straightforward way)
\begin{eqnarray}\label{Rm-evo-equ}
\Delta_f\arrowvert\nabla^k\Rm(g)\arrowvert^2&\geq&2\arrowvert\nabla^{k+1}\Rm(g)\arrowvert^2-2\left(1+\frac{k}{2}+c(n,k)\arrowvert\Rm(g)\arrowvert\right)\arrowvert\nabla^k\Rm(g)\arrowvert^2\\
&&-c(n,k)\sum_{i=1}^{k-1}\arrowvert\nabla^i\Rm(g)\arrowvert\arrowvert\nabla^{k-i}\Rm(g)\arrowvert\arrowvert\nabla^k\Rm(g)\arrowvert
\end{eqnarray}
By applying the induction assumption, we get
\begin{eqnarray*}
&&\Delta_f\arrowvert\nabla^k\Rm(g)\arrowvert^2\geq2\arrowvert\nabla^{k+1}\Rm(g)\arrowvert^2-2\left(1+\frac{k}{2}+c(n,k)\arrowvert\Rm(g)\arrowvert\right)\arrowvert\nabla^k\Rm(g)\arrowvert^2\\
&&-C\left(n,k,\min_{M^n}v,\sup_{M^n}\arrowvert\Rm(g)\arrowvert,\A_g^0(\Rm(g)),...,\A_g^{k-1}(\Rm(g))\right)v^{-(k/2+2)}\arrowvert\nabla^k\Rm(g)\arrowvert.
\end{eqnarray*}
By considering $u^k_{\epsilon}:=\sqrt{\arrowvert\nabla^k\Rm(g)\arrowvert^2+\epsilon^2}$, one has
\begin{eqnarray*}
\Delta_fu^k_{\epsilon}&\geq&-\left(1+\frac{k}{2}+c(n,k)\arrowvert\Rm(g)\arrowvert\right)u_{\epsilon}^k\\
&&-C\left(n,k,\min_{M^n}v,\sup_{M^n}\arrowvert\Rm(g)\arrowvert,\A_g^0(\Rm(g)),...,\A_g^{k-1}(\Rm(g))\right)v^{-(k/2+2)}.
\end{eqnarray*}
Since 
\begin{eqnarray*}
\Delta_f(v^{k/2+1})=\left(1+\frac{k}{2}\right)v^{k/2+1}+\frac{k}{2}\left(1+\frac{k}{2}\right)\arrowvert\nabla\ln v\arrowvert^2v^{k/2+1},
\end{eqnarray*}
we get, by corollary \ref{a-prio-bd-curv-tensor},
\begin{eqnarray*}
\Delta_f\left(u_{\epsilon}^kv^{k/2+1}\right)&\geq& 2\left(1+\frac{k}{2}\right)<\nabla(u_{\epsilon}^kv^{k/2+1}),\nabla\ln v>\\
&&-C(n,k,\sup_{M^n}\arrowvert\Rm(g)\arrowvert,\A_g^0(\Rm(g)))v^{-1}\left(u_{\epsilon}^kv^{k/2+1}\right)\\
&&-C\left(n,k,\min_{M^n}v,\sup_{M^n}\arrowvert\Rm(g)\arrowvert,\A_g^0(\Rm(g)),...,\A_g^{k-1}(\Rm(g))\right)v^{-1}.
\end{eqnarray*}
Now,
\begin{eqnarray*}
&&\Delta_{v-(2+k)\ln v}e^{-C_0v^{-1}}=\frac{C_0}{v^2}\left(v-(2+k)\frac{\arrowvert \nabla v\arrowvert^2}{v}\right)+C_0^2\arrowvert\nabla v^{-1}\arrowvert^2,\\
&&\Delta_{v-(2+k)\ln v}\left(u_{\epsilon}^kv^{k/2+1}e^{-C_0v^{-1}}\right)\geq-2C_0<\nabla\left(u_{\epsilon}^kv^{k/2+1}e^{-C_0v^{-1}}\right),\nabla v^{-1}>\\
&&+\left(\frac{C_0}{2}-C(n,k,\sup_{M^n}\arrowvert\Rm(g)\arrowvert,\A_g^0(\Rm(g)))-C_0^2v^{-2}\right)\cdot v^{-1}\cdot\left(u_{\epsilon}^kv^{k/2+1}e^{-C_0v^{-1}}\right)\\
&&-C\left(n,k,\min_{M^n}v,\sup_{M^n}\arrowvert\Rm(g)\arrowvert,\A_g^0(\Rm(g)),...,\A_g^{k-1}(\Rm(g))\right)e^{-C_0v^{-1}}v^{-1}\\
&\geq&-2C_0<\nabla\left(u_{\epsilon}^kv^{k/2+1}e^{-C_0v^{-1}}\right),\nabla v^{-1}>\\
&&-C\left(n,k,\min_{M^n}v,\sup_{M^n}\arrowvert\Rm(g)\arrowvert,\A_g^0(\Rm(g)),...,\A_g^{k-1}(\Rm(g))\right)e^{-C_0v^{-1}}v^{-1},
\end{eqnarray*}
for 
\begin{eqnarray*}
&&C_0=C\left(n,k,\sup_{M^n}\arrowvert\Rm(g)\arrowvert,\A_g^0(\Rm(g))\right),\\ 
&&v\geq C:=C\left(n,k,\sup_{M^n}\arrowvert\Rm(g)\arrowvert,\A_g^0(\Rm(g))\right).
\end{eqnarray*}
Finally, for any constant $C_1$ larger than $C$, one has
\begin{eqnarray*}
\Delta_{v-(2+k)\ln v+2C_0v^{-1}}\left(u_{\epsilon}^kv^{k/2+1}e^{-C_0v^{-1}}-C_1v^{-1}\right)\geq 0,
\end{eqnarray*}
for $v\geq C\left(n,k,\sup_{M^n}\arrowvert\Rm(g)\arrowvert,\A_g^0(\Rm(g))\right)$. Therefore, by the maximum principle and the fact that these estimates does not depend on $\epsilon\leq1$, one gets the expected estimates since one can assume that $$u_{\epsilon}^kv^{k/2+1}e^{-C_0v^{-1}}-C_1v^{-1}\leq 0$$ on $v\leq C(n,k,\sup_{M^n}\arrowvert\Rm(g)\arrowvert,\A_g^0(\Rm(g)))$.

\end{proof}

We are now in a position to prove theorem \ref{Compactness-I}.

\begin{proof}[Proof of theorem \ref{Compactness-I}]
\begin{itemize}

\item  Firstly, we remark that if $(M^n,g,\nabla f,p)\in\cat{M}^0_{Exp}(n,\lambda_0,(\Lambda_k)_{k\geq 0})$ then $(M^n,g,\nabla f)$ is asymptotically conical to $(C(X),g_{C(X)},r\partial_r/2)$ by theorem \ref{asy-con-gen}. Moreover, as $\Ric(g)\geq 0$, $X$ is diffeomorphic to $\mathbb{S}^{n-1}$. We claim that $$\AVR(g)\geq C(n,\A^0_g(\Rm(g)))\geq C(n,\Lambda_0)>0.$$

Indeed, $(X,g_X)$ is a smooth Riemannian manifold such that 
\begin{eqnarray*}
&&\Ric(g_X)\geq(n-2)g_X,\\
&&\sup_{X}\arrowvert\Rm(g_X)\arrowvert\leq C(n,\A_g^0(\Rm(g)))\leq C(n,\Lambda_0).
\end{eqnarray*}

On the one hand, if $n$ is odd, then the Gauss-Bonnet applied to $(X,g_X)$ gives
\begin{eqnarray*}
2=\chi(X)=\frac{2}{\vol(\mathbb{S}^{n-1})}\int_{X}\textbf{K},
\end{eqnarray*}
where 
\begin{eqnarray*}
\textbf{K}=
\frac{1}{(n-1)!}\sum_{i_1<...<i_{n-1}}\epsilon_{i_1,...,i_{n-1}}\Rm(g_X)_{i_1,i_2}\wedge...\wedge\Rm(g_X)_{i_{n-2},i_{n-1}},
\end{eqnarray*}
where $\epsilon_{i_1,...i_{n-1}}$ is the signature of the permutation $(i_1,...,i_{n-1})$. Therefore,
\begin{eqnarray*}
0<C(n)\leq \vol_{g_X}(X)(C(n)+\A^0_g(\Rm(g))^{\frac{n-1}{2}}\leq C(n,\Lambda_0)\AVR(g),
\end{eqnarray*}
which proves the claim if $n$ is odd.

On the other hand, if $n$ is even, we apply the $\pi_2$-theorem of \cite{Pet-Tus-Pi_2} asserting in our setting that there exists a positive constant $V_0=V_0(n,\Lambda_0)$ such that the volume of any simply connected compact Riemannian manifold $(X,g_X)$ with finite second homotopy group and such that $\Ric(g_X)\geq c(n)g_{X}>0$, $\Rm(g_X)\leq C(n,\Lambda_0)$ is bounded from below by $V_0$.

As a consequence, any link of the asymptotic cones of such Ricci expanders belongs to $ \cat{M}(n-1,\pi,C(n,\Lambda_0),(\widetilde{\Lambda_k})_{ k\geq0})$ where $\widetilde{\Lambda_k}$ depends on a finite number of $\Lambda_i$. Indeed, the diameter estimate follows by the Myers theorem and the curvature estimates follow by Gauss equations. Hence, these metrics lie in a compact set for the $C^{\infty}$ topology according to theorem \ref{comp-ham-static}.\\

\item Thanks to proposition \ref{pot-fct-est}, we notice that
\begin{eqnarray}
\arrowvert\min_{M^n}f\arrowvert+\arrowvert\mu(g)\arrowvert\leq C(n,\lambda_0,\Lambda_0).\label{bd-pot-fct-ent-cpt}
\end{eqnarray}

Therefore, if $(M_i^n,g_i,\nabla f_i,p_i)_i\in \cat{M}^0_{Exp}(n,\lambda_0,(\Lambda_k)_{k\geq 0})$, then $(M_i^n,g_i,p_i)_i$ belongs to 
\begin{eqnarray*}
\cat{M}_{Pt}(n,v,(C(n,k,\lambda_0))_{k\geq0})&:=&\{(M^n,g,p) \quad \mbox{complete}\quad | \vol_gB_g(p,1)\geq v\quad;\\
&& \arrowvert\nabla^k\Rm(g)\arrowvert\leq C(n,k,\lambda_0),\forall k\in\mathbb{N}\},
 \end{eqnarray*}
 where $v=v(n,\Lambda_0)$, which is compact for the pointed $C^{\infty}$ topology. Indeed, the bounds on the covariant derivatives of the curvature tensor come from lemma \ref{elliptic} and the volume estimate is due to the Bishop-Gromov theorem together with the estimate $\AVR(g)\geq C(n,\Lambda_0)$. In particular, $(M_i,g_i,p_i)_i$ (sub)converges to a smooth Riemannian manifold $(M_{\infty},g_{\infty},p_{\infty})\in\cat{M}_{Pt}(n,v,(C(n,k,\lambda_0))_{k\geq0})$ with non negative Ricci curvature. Moreover, by the Ricci soliton equation together with the bounds (\ref{bd-pot-fct-ent-cpt}), $(f_i)_i$ subconverges to a smooth function $f_{\infty}$ satisfying the Ricci soliton as well : $f_{\infty}$ is a smooth strictly convex function with $\crit(f_{\infty})=\{p_{\infty}\}$. We also check that $f_{\infty}$ is normalized. Indeed, for a positive radius $R>0$ and some index $i$, we have
 \begin{eqnarray*}
 \left\arrowvert(4\pi)^{n/2}-\int_{B(p_i,R)}e^{-f_i}d\mu(g_i)\right\arrowvert&\leq&\int_{r_{p_i}\geq R}e^{-f_i}d\mu(g_i)\\
 &\leq&e^{-\min_{M_i^n}f_i}\int_{r_{p_i}\geq R}e^{-r_{p_i}^2/4}d\mu(g_i)\\
 &\leq&C(n,\lambda_0,\Lambda_0)\int_R^{+\infty}e^{-r^2/4}r^{n-1}dr.
 \end{eqnarray*}
 If $R$ is fixed and if $i$ tends to $+\infty$, 
 \begin{eqnarray*}
\left\arrowvert(4\pi)^{n/2}-\int_{B(p_{\infty},R)}e^{-f_{\infty}}d\mu(g_{\infty})\right\arrowvert\leq C(n,\lambda_0,\Lambda_0)\int_R^{+\infty}e^{-r^2/4}r^{n-1}dr, 
\end{eqnarray*}
which implies in particular that $f_{\infty}$ is normalized.\\

\item Now, we justify the inversion of limits. Let $(M_i^n,g_i,\nabla f_i,p_i)_i\in \cat{M}^0_{Exp}(n,\lambda_0,(\Lambda_k)_{k\geq 0})$ be converging to a normalized expanding gradient Ricci soliton $(M_{\infty},g_{\infty},\nabla f_{\infty},p_{\infty})$ with nonnegative Ricci curvature. Let $(C(X_i),dr^2+r^2g_{X_i})_i$ be the sequence of asymptotic cones corresponding to $(M_i,g_i,p_i)_i$. As noticed previously, the sequence of compact Riemannian manifolds $(X_i,g_{X_i})_i$ has a subsequence converging to a smooth Riemannian manifold $(X_{\infty},g_{X_{\infty}})$. On the other hand, according to proposition \ref{curv-op-max-ppe-infty} together with the crucial estimate (\ref{est-rate-con-exp}) given in the proof of theorem \ref{asy-con-gen}, one has, 
\begin{eqnarray}
\partial_t\nabla^{g_{X_i},k}\bar{g_i}_t&=&\textit{O}(t^{-3}),\quad\forall k\geq 0,
\end{eqnarray}
where $\textit{O}$ depends on $k$ and is uniform in the indices $i$. In particular,
\begin{eqnarray*}
\nabla^{g_{X_i},k}((\bar{g_i})_t-g_{X_i})=\textit{O}(t^{-2}),
\end{eqnarray*}
which proves that the asymptotic cone of $(M_{\infty},g_{\infty},p_{\infty})$ is isometric to $(C(X_{\infty}),dr^2+r^2g_{X_{\infty}})$.

The last step is to make sure that the invariants $(\A^k_{g_{\infty}}(\Rm(g_{\infty}))_{k\geq 0}$ have not improved. Indeed,
\begin{eqnarray*}
\A^k_{g_{\infty}}(\Rm(g_{\infty}))=\lim_{i\rightarrow+\infty}\limsup_{x\rightarrow +\infty}d_{g_i}^{2+k}(p_i,x)\arrowvert\nabla^{k}\Rm(g_i)\arrowvert\leq \Lambda_k.
\end{eqnarray*}

\end{itemize}
\end{proof}

\begin{rk}
What if one only assumes bounds on the asymptotic covariant derivatives of the Ricci curvature $(\A_g^k(\Ric(g)))_{k\geq 0}$ ? This question is motivated by theorem \ref{asy-con-gen} where it is proved that the full curvature tensor is actually asymptotically controlled by the Ricci curvature. 
\end{rk}

We state and prove another compactness theorem for Ricci expander with nonnegative curvature operator : the bound on the curvature is replaced by a lower bound on the asymptotic volume ratio.

\begin{theo}\label{Compactness-II}
The class
\begin{eqnarray*}
\cat{M}^{\vol}_{Exp}(n,(\Lambda_k)_{k\geq 0},V_0)&:=&\{\mbox{$(M^n,g,\nabla f,p)$ normalized expanding gradient Ricci soliton. s.t.} 
\\&&\Rm(g)\geq 0\quad;\quad\crit(f)=\{p\}\quad;\\
&&\quad\AVR(g)\geq V_0\quad ;\quad \A_g^k(\Rm(g))\leq\Lambda_k,\quad\forall k\geq 0 \}
\end{eqnarray*}
is compact in the pointed $C^{\infty}$ topology. \\

Moreover, let a sequence $(M_i,g_i,\nabla f_i,p_i)_i$ be in $\cat{M}^{\vol}_{Exp}(n,(\Lambda_k)_{k\geq 0},V_0)$. Then there exists a subsequence converging in the pointed $C^{\infty}$ topology to an expanding gradient Ricci soliton $(M_{\infty},g_{\infty},\nabla f_{\infty},p_{\infty})$ in $\cat{M}^{\vol}_{Exp}(n,(\Lambda_k)_{k\geq 0},V_0)$ whose asymptotic cone\\ $(C(X_{\infty}),g_{C(X_{\infty})},o_{\infty})$ is the limit in the Gromov-Hausdorff topology of the sequence of the asymptotic cones $(C(X_i),g_{C(X_i)},o_i)_i$ with $C^{\infty}$ convergence outside the apex. 
\end{theo}
\begin{rk}
\begin{itemize}
\item If one is only interested in $C^{1,\alpha}$ convergence in the preceding theorem, assuming $\A^0(\R_g)\leq \Lambda_0$ is enough to ensure such convergence since the scalar curvature controls pointwise the curvature operator.
\item Again, compared to theorem \ref{Compactness-I}, theorem \ref{Compactness-II} does not assume any a priori bound on the curvature tensor : this is actually a consequence as shown in the proof below.
\item The proof of theorem \ref{Compactness-II} is more self-contained than the one of theorem \ref{Compactness-II} since it does not use the involved non collapsing theory developed in \cite{Pet-Tus-Pi_2}.
\end{itemize}
\end{rk}
\begin{proof}
By Theorem \ref{Compactness-I}, it suffices to prove that if $(M^n,g,\nabla f,p)$ belongs to $\cat{M}^{\vol}_{Exp}(n,(\Lambda_k)_{k\geq0},V_0)$ then 
\begin{eqnarray*}
\sup_{M^n}\arrowvert\Rm(g)\arrowvert\leq C(n,V_0).
\end{eqnarray*}

This estimate has already been proved in greater generality by Schulze and Simon \cite{Sch-Sim}. Nonetheless we give here a self-contained proof. This goes by contradiction. Assume there exists a sequence of expanding gradient Ricci solitons $(M_i^n,g_i,\nabla f_i,p_i)_i$ such that 
\begin{eqnarray*}
\Rm(g_i)\geq 0\quad;\quad\AVR(g_i)\geq V_0\quad;\quad\crit(f_i)=\{p_i\}\quad;\quad\sup_i\sup_{M_i^n}\R_{g_i}=+\infty.
\end{eqnarray*} 

As the Ricci curvatures of the metrics $g_i$ are nonnegative, the scalar curvatures $\R_{g_i}$ attain their maximum at the (unique) point $p_i\in M_i^n$. Moreover, we have the following estimates of the potential functions by proposition \ref{pot-fct-est} :
\begin{eqnarray}\label{pot-fct-seq}
\frac{1}{4}r_{p_i}(x)^2+\min_{M_i^n}v_i\leq v_i(x)\leq\left(\frac{1}{2}r_{p_i}(x)+\sqrt{\min_{M_i^n}v_i}\right)^2,
\end{eqnarray}
for any $x\in M_i^n$. By (\ref{equ:2}), we have $\min_{M_i^n}v_i=\max_{M_i^n}\R_{g_i}+n/2=\R_{g_i}(p_i)+n/2$. Define the rescaled metrics $\tilde{g_i}(\tau):=Q_ig_i(Q_i^{-1}\tau)$ where $Q_i:=\R_{g_i}(p_i)+n/2$ and $g_i(\cdot)$ is the associated Ricci flow to the expanding gradient Ricci soliton $(M_i^n,g_i,\nabla f_i)$. Then the sequence $(M_i^n,\tilde{g_i}(\tau))_{\tau\in(-Q_i,+\infty)}$ of Ricci flows satisfies : $\sup_{M^n}\arrowvert\Rm(\tilde{g_i}(\tau))\arrowvert$ is uniformly bounded in the indices $i$ for $\tau$ in a compact interval of $(-Q_i,+\infty)$, $\R_{\tilde{g_i}(0)}(p_i)=1$, $\Rm(\tilde{g_i}(\tau))\geq 0$ and $\AVR(\tilde{g_i}(\tau))\geq V_0$ for $\tau \in (-Q_i,+\infty)$ which means in particular that $\inj_{\tilde{g_i}(0)}(p_i)\geq C(n,V_0)>0$.  Therefore, by Hamilton's compactness theorem \cite{Ham-Com}, there is a subsequence, still denoted by $(M_i^n,\tilde{g_i}(\tau),p_i)_i$ converging in the pointed Cheeger-Gromov topology to a non flat eternal solution of the Ricci flow $(M_{\infty}^n,g_{\infty}(\tau),p_{\infty})$ with nonnegative curvature operator and Euclidean volume growth, i.e. such that $\AVR(g_{\infty})>0$.
 Moreover, $\tilde{v_i}:=v_i-\min_{M_i^n}v_i$ satisfy by the estimates (\ref{pot-fct-seq}) : 
 \begin{eqnarray*}
 0\leq\frac{r_{\tilde{g_i}(0),p_i}^2}{4Q_i}\leq\tilde{v_i}\leq\frac{r_{\tilde{g_i}(0),p_i}^2}{4Q_i}+r_{\tilde{g_i}(0),p_i}\quad;\quad \tilde{v}_i(p_i)=0.
 \end{eqnarray*}
 Moreover, by (\ref{equ:2}), $$\arrowvert\nabla^{\tilde{g_i}(0)}\tilde{v_i}\arrowvert_{\tilde{g_i}(0)}\leq Q_i^{-1}\left(\frac{r_{\tilde{g_i}(0),p_i}^2}{4Q_i}+r_{\tilde{g_i}(0),p_i}+Q_i\right),$$ which means in particular that the sequence of functions $(\tilde{v_i})_i$ is equicontinuous. Finally, the soliton equation can be rewritten as 
 \begin{eqnarray*}
\nabla^{2,\tilde{g_i}(0)}\tilde{v_i}=\Ric(\tilde{g_i}(0))+\frac{\tilde{g_i}(0)}{2Q_i}.
\end{eqnarray*}
Therefore, the higher covariant derivatives of $\tilde{v_i}$ are uniformly bounded and $\tilde{v_i}$ converges smoothly to a smooth function $v_{\infty}$ which satisfies $$\nabla^{2,g_{\infty}}v_{\infty}=\Ric(g_{\infty}),$$
i.e. $(M_{\infty},g_{\infty},\nabla v_{\infty})$ is a non flat steady gradient Ricci soliton with nonnegative curvature operator and positive asymptotic volume ratio : a contradiction to 
 Hamilton's argument [Chap.$9$, Sec. $3$, \cite{Ben}] which can be proved by induction on the dimension.

\end{proof}

\bibliographystyle{alpha.bst}
\bibliography{bib-comp-egs}

\end{document}